\documentclass{amsart}
\usepackage{amsmath}
\usepackage{amsthm}
\usepackage{amssymb}
\usepackage[margin=0.5in]{caption}
\usepackage{comment}
\usepackage{dsfont}
\usepackage{enumitem}
\usepackage{graphicx}
\usepackage[hidelinks]{hyperref}
\usepackage[capitalise]{cleveref}
\usepackage{mathrsfs}
\usepackage{mathtools}
\usepackage{nameref}
\usepackage[new]{old-arrows}
\usepackage{seqsplit}
\usepackage{stmaryrd}
\usepackage{thmtools}
\usepackage{url}
\usepackage{xcolor}
\usepackage{xparse}

\usepackage{tikz}
\usetikzlibrary{cd}
\usetikzlibrary{positioning, angles, quotes}
\newcommand{\nospacepunct}[1]{\makebox[0pt][l]{\,#1}} 

\theoremstyle{plain}
\newtheorem{thm}{Theorem}[section]
\newtheorem*{thm*}{Theorem}

\newtheorem{cor}[thm]{Corollary}
\newtheorem*{cor*}{Corollary}

\newtheorem{prop}[thm]{Proposition}
\newtheorem*{prop*}{Proposition}

\newtheorem{lem}[thm]{Lemma}
\newtheorem*{lem*}{Lemma}

\newtheorem*{claim*}{Claim}

\newtheorem*{exer*}{Exercise}

\newtheorem{conj}[thm]{Conjecture}
\newtheorem*{conj*}{Conjecture}

\theoremstyle{definition}
\newtheorem{defn}[thm]{Definition}
\newtheorem*{defn*}{Definition}

\newtheorem*{ex*}{Example}

\newtheorem*{q*}{Question}

\theoremstyle{remark}
\newtheorem{rem}[thm]{Remark}
\newtheorem*{rem*}{Remark}

\theoremstyle{plain}

\Crefname{thm}{Theorem}{Theorems}
\Crefname{defn}{Definition}{Definitions}
\Crefname{lem}{Lemma}{Lemmata}

\newcommand{\VN}{\mathcal{N}}
\newcommand{\UO}{\mathcal{U}}

\newcommand{\C}{\mathbb{C}}

\newcommand{\Q}{\mathbb{Q}}

\newcommand{\Z}{\mathbb{Z}}

\newcommand{\inv}{^{-1}}

\DeclareMathOperator{\coker}{coker}

\DeclareMathOperator{\im}{Im}

\DeclareMathOperator{\lcm}{lcm}

\DeclareMathOperator{\Mat}{Mat}
\DeclareMathOperator{\Ore}{Ore}
\DeclareMathOperator{\pr}{pr}

\DeclareMathOperator{\rk}{rk}
\DeclareMathOperator{\SMRF}{SMRF}

\DeclareMathOperator{\Tor}{Tor}
\DeclareMathOperator{\trace}{tr}

\newcommand{\vertii}[1]{{\left\vert\kern-0.25ex\left\vert #1 \right\vert\kern-0.25ex\right\vert}}
\newcommand{\vertiii}[1]{{\left\vert\kern-0.25ex\left\vert\kern-0.25ex\left\vert #1 \right\vert\kern-0.25ex\right\vert\kern-0.25ex\right\vert}}

\makeatletter
\newsavebox{\@brx}
\newcommand{\llangle}[1][]{\savebox{\@brx}{\(\m@th{#1\langle}\)}%
  \mathopen{\copy\@brx\mkern2mu\kern-0.9\wd\@brx\usebox{\@brx}}}
\newcommand{\rrangle}[1][]{\savebox{\@brx}{\(\m@th{#1\rangle}\)}%
  \mathclose{\copy\@brx\mkern2mu\kern-0.9\wd\@brx\usebox{\@brx}}}
\makeatother

\title[Universal localizations, Atiyah conjectures and graphs of groups]{Universal localizations, Atiyah conjectures and graphs of groups}

\author{Pablo S\'anchez-Peralta}
\address[P.~ S\'anchez-Peralta]{Universidad Aut\'onoma de Madrid, Madrid, Spain}
\email{pablo.sanchezperalta@uam.es}

\begin{document}

\begin{abstract}
    Let $G$ be a countable group that is the fundamental group of a graph of groups with finite edge groups and vertex groups satisfying the strong Atiyah conjecture over $K \subseteq \C$ a field closed under complex conjugation. Assume that the orders of finite subgroups of $G$ are bounded above. We show that $G$ satisfies the strong Atiyah conjecture over $K$. In particular, this implies that the strong Atiyah conjecture is closed under free products. Moreover, we prove that the $\ast$-regular closure of $K[G]$ in $\UO(G)$, $\mathcal{R}_{\scalebox{0.8}{$K[G]$}}$, is a universal localization of the graph of rings associated to the graph of groups, where  the rings are the corresponding $\ast$-regular closures. As a result, we obtain that the algebraic and center-valued Atiyah conjecture over $K$ are also closed under the graph of groups construction as long as the edge groups are finite.  We also infer some consequences on the structure of the $K_0$ and $K_1$-groups of $\mathcal{R}_{\scalebox{0.8}{$K[G]$}}$. The techniques developed enable us to prove that $K[G]$ fulfills the strong, algebraic and center-valued Atiyah conjectures, and that $\mathcal{R}_{\scalebox{0.8}{$K[G]$}}$ is the universal localization of $K[G]$ over the set of all matrices that become invertible in $\UO(G)$, provided that $G$ belongs to a certain class of groups $\mathcal{T}_{\scalebox{0.7}{$\mathcal{VLI}$}}$, which contains in particular virtually-\{locally indicable\} groups that are the fundamental group of a graph of virtually free groups.
\end{abstract}

\maketitle

\section{Introduction} \label{sec: introd}

    Let $G$ be a countable group and assume that the orders of finite subgroups of $G$ are bounded above. We denote by $\lcm(G)$ the least common multiple of the orders of finite subgroups of $G$. The {\it strong Atiyah conjecture for $G$ over a subfiled $K \subseteq \C$} predicts that the $L^2$-dimension function \cite[Definition 8.28]{Luck02} over finitely presented left $K[G]$-modules takes values in $\frac{1}{\lcm(G)} \Z$. This is an algebraic reformulation of the classical statement in which a group $G$ acting freely and cocompactly on a CW complex $X$ satisfies the strong Atiyah conjecture over $\Q$ if the $L^2$-Betti numbers $\beta_i^{(2)}(X,G)$ belong to $\frac{1}{\lcm(G)} \Z$. It is worth mentioning that without the assumption of bounded torsion, the strong Atiyah conjecture is false in general by work of Grigorchuk--Linnell--Schick--\.Zuk \cite[Theorem 2]{GLSZnotbddAtiyah}.

    The conjecture is now known for several classes of groups, for instance Linnell's class $\mathcal{C}$ \cite[Theorem 1.5]{LinnellDivRings93} that contains all free-by-{elementary amenable} groups as well as fundamental groups of three-manifolds \cite[Theorem 1.4]{KielakLinton_3mfldAtiyah}, Schick's class $\mathcal{D}$ \cite[Theorem 1]{SchickIntegrality} that includes all residually-\{torsion-free elementary amenable\} groups, see \cite[Corollary 1.2]{JaikinBaseChange}, or locally indicable groups by Jaikin-Zapirain and L\'opez-\'Alvarez \cite[Theorem 1.1]{JaikinLopez_Atiyah} which encompasses torsion-free one-relator groups. For a detailed overview of the current status of the strong Atiyah conjecture and its inheritance properties see \cite[Section 2.4.3]{FabianHennekeThesis}. Although the three classes of groups that we have just mentioned are closed under free products, almost nothing has been known about the general question of whether the strong Atiyah conjecture passes to free products of groups or not. In this article, we settle this question and examine how the strong Atiyah conjecture behaves with respect to the graph of groups construction.

    \begin{thm}\label{thm: SAC_GG}
        Let $\mathscr G_\Gamma = (G_v, G_e)$ be a graph of groups with countable fundamental group $G$. Assume that every vertex group $G_v$ satisfies the strong Atiyah conjecture over $K \subseteq \C$ a field closed under complex conjugation and that every edge group $G_e$ is finite. Assume in addition that $\lcm(G)$ is finite. Then $G$ satisfies the strong Atiyah conjecture over $K$.
    \end{thm}

    Given a graph of groups $\mathscr{G}_\Gamma = (G_v, G_e)$ with fundamental group $G$, we can always consider the associated graph of rings $\mathscr{RG}_\Gamma = (\mathcal{R}_{\scalebox{0.8}{$K[G_v]$}}, \mathcal{R}_{\scalebox{0.8}{$K[G_e]$}})$ of its $\ast$-regular closures (see \cref{subsec: vNdim_reg_clos} for the definition) where the ring homomorphisms are induced from the group homomorphisms in the graph of groups. Note that there is a natural ring homomorphism $\phi: \mathscr{RG}_\Gamma \rightarrow \mathcal{R}_{\scalebox{0.8}{$K[G]$}}$ which endows $\mathscr{RG}_\Gamma$ with a Sylvester module rank function (\cref{def: SMRF}), abbreviated as $\SMRF$, via a pull-back of the $L^2$-dimension. From an algebraic point of view, the strong Atiyah conjecture determines the ring structure of $\mathcal{R}_{\scalebox{0.8}{$K[G]$}}$, namely as a semisimple ring, which is in fact a division ring if $G$ is torsion-free. Note that the latter is a stronger version of the Kaplansky zero-divisor conjecture for $K[G]$. We go further and prove that $\mathcal{R}_{\scalebox{0.8}{$K[G]$}}$ possesses a very specific structure: that of a {\it universal localization of} $\mathscr{RG}_\Gamma$, meaning that $\mathcal{R}_{\scalebox{0.8}{$K[G]$}}$ is universal among the $\mathscr{RG}_\Gamma$-rings with respect to the property that every full map over $\mathscr{RG}_\Gamma$ has an inverse in $\mathcal{R}_{\scalebox{0.8}{$K[G]$}}$ (see \cref{subsec: univ_local} for more details). As a matter of fact, the universal localization structure of $\mathcal{R}_{\scalebox{0.8}{$K[G]$}}$ enables us to understand its low degree $K$-theory, that is, the $K_0$ and $K_1$-groups.

    \begin{thm}\label{thm: GG_univ_local_&_K_0}
        Let $\mathscr G_\Gamma = (G_v, G_e)$ be a graph of groups with countable fundamental group $G$. Assume that every vertex group $G_v$ satisfies the strong Atiyah conjecture over $K \subseteq \C$ a field closed under complex conjugation and that every edge group $G_e$ is finite. Assume in addition that $\lcm(G)$ is finite. Then the following holds:
        \begin{enumerate}[label=(\roman*)]
            \item The map $\oplus_{v \in V} K_0(\mathcal{R}_{\scalebox{0.8}{$K[G_v]$}}) \rightarrow K_0(\mathscr{RG}_\Gamma)$ is surjective;
            \item the $\SMRF$ induced from $\UO(G)$ via $\phi$ is faithful over finitely generated projective left $\mathscr{RG}_\Gamma$-modules;
            \item $\mathcal{R}_{\scalebox{0.8}{$K[G]$}}$ is the universal localization of all full maps over $\mathscr{RG}_\Gamma$ with respect to the $L^2$-dimension; and
            \item the map $\oplus_{v \in V} K_0(\mathcal{R}_{\scalebox{0.8}{$K[G_v]$}}) \rightarrow K_0(\mathcal{R}_{\scalebox{0.8}{$K[G]$}})$ is surjective.
        \end{enumerate}
    \end{thm}

    In the survey article \cite{Jaikin_l2survey}, Jaikin-Zapirain introduces a variation of the strong Atiyah conjecture in terms of the finitely generated projective left modules of $\mathcal{R}_{\scalebox{0.8}{$K[G]$}}$, the so-called {\it algebraic Atiyah conjecture}. This conjecture predicts that given a countable group $G$ with finite $\lcm(G)$ and $K\subseteq \C$ a field closed under complex conjugation, the map 
    \[
    \bigoplus_{F \leqslant G, |F|<\infty} K_0(K[F]) \rightarrow K_0(\mathcal{R}_{\scalebox{0.8}{$K[G]$}})
    \]
    is surjective. Note that as a consequence of \cref{thm: GG_univ_local_&_K_0} we can deduce that the algebraic Atiyah conjecture is also well-behaved under the graph of groups construction.

    \begin{cor}\label{cor: AAC_GG}
        Let $\mathscr G_\Gamma = (G_v, G_e)$ be a graph of groups with countable fundamental group $G$. Assume that every vertex group $G_v$ satisfies the algebraic Atiyah conjecture over $K \subseteq \C$ a field closed under complex conjugation and that every edge group $G_e$ is finite. Assume in addition that $\lcm(G)$ is finite. Then $G$ satisfies the algebraic Atiyah conjecture over $K$.
    \end{cor}

    There is another refinement of the Atiyah conjecture known as the \textit{center-valued Atiyah conjecture}. Just as $\dim_{\scalebox{0.8}{$\VN(G)$}}$ is obtained from the von Neumann trace $\trace_{\scalebox{0.8}{$\VN(G)$}}$, one can construct generalized dimension functions out of trace-like functions taking values over the center $\mathcal{Z}(\VN(G))$. Fortunately, among all of them there is a universal choice, termed the \textit{center-valued trace} $\trace^u_{\scalebox{0.8}{$\VN(G)$}}$, in the sense that all possible traces on $\VN(G)$ taking values in the center $\mathcal{Z}(\VN(G))$ factor through it cf. \cite[Theorem 9.5]{Luck02}. Denote by $\dim^u_{\scalebox{0.8}{$\VN(G)$}}$ its associated dimension function. The center-valued Atiyah conjecture claims that given a countable group $G$ with finite $\lcm(G)$ and $K\subseteq \C$ a field closed under complex conjugation, for every finitely presented left $K[G]$-module $M$ it holds that
    \[
    \dim^u_{\scalebox{0.8}{$\VN(G)$}}(\VN(G) \otimes_{\scalebox{0.8}{$K[G]$}} M) \in L_K(G),
    \]
    where $L_K(G)$ is the additive subgroup of $\mathcal{Z}(\VN(G))$ generated by 
    \[
    \{ \trace^u_{\scalebox{0.8}{$\VN(G)$}}(p) : p \in K[F], p = p^2 = p^{\ast}, F \leqslant G, |F| < \infty \}.
    \]
    From the work of Knebusch, Linnell and Schick \cite[Theorem 3.7]{KLS_centervaluedA}, Henneke showed in \cite[Theorem 3.14]{FabianHennekeThesis} that the center-valued Atiyah conjecture is equivalent to the algebraic Atiyah conjecture. Consequently, we immediately get the same permanence result for the center-valued Atiyah conjecture.

    \begin{cor}
        Let $\mathscr G_\Gamma = (G_v, G_e)$ be a graph of groups with countable fundamental group $G$. Assume that every vertex group $G_v$ satisfies the center-valued Atiyah conjecture over $K \subseteq \C$ a field closed under complex conjugation and that every edge group $G_e$ is finite. Assume in addition that $\lcm(G)$ is finite. Then $G$ satisfies the center-valued Atiyah conjecture over $K$.
    \end{cor}

    Given a group $G$, let $\Sigma(K,G)$ denote the set of matrices over $K[G]$ that are invertible over $\UO(G)$. Reich showed in \cite[Theorem 8.3]{ReichThesis} that $\mathcal{R}_{\scalebox{0.8}{$K[G]$}}$ is the universal localization of $K[G]$ over the set $\Sigma(K,G)$ whenever $G$ lies in Linnell's class $\mathcal{C}$ and has finite $\lcm(G)$. Here we prove the same result for another class of groups.

    \begin{defn}
        Let $\mathcal{T}$ denote the smallest class of groups which
        \begin{enumerate}[label=(\roman*)]
            \item contains all finite groups;
            \item is closed under directed unions; and
            \item satisfies $G \in \mathcal{T}$ whenever $G$ is the fundamental group of a graph of groups in $\mathcal{T}$.
        \end{enumerate}
    \end{defn}

    The class $\mathcal{T}$ should enjoy the several Atiyah conjectures and the localization property; however, the techniques developed in the present paper do not seem to be enough. To this end, we shall restrict to the subclass $\mathcal{T}_{\scalebox{0.7}{$\mathcal{VLI}$}}$ where \textit{(iii)} is only allowed provided that $G$ is in addition virtually-\{locally indicable\}.

    \begin{thm}\label{thm: class_TVLI}
        Let $G$ be a countable group in the class $\mathcal{T}_{\scalebox{0.7}{$\mathcal{VLI}$}}$ with finite $\lcm(G)$ and $K \subseteq \C$ a field closed under complex conjugation. Then the following holds:
        \begin{enumerate}[label=(\roman*)]
            \item The ring $\mathcal{R}_{\scalebox{0.8}{$K[G]$}}$ is semisimple and is the universal localization of $K[G]$ over $\Sigma(K,G)$; and
            \item $G$ satisfies the algebraic Atiyah conjecture over $K$. In particular, G also satisfies the strong and center-valued Atiyah conjecture over $K$.
        \end{enumerate}
    \end{thm}

    In \cite{LuckLinnell_localization}, Linnell and L\"uck study a variation of the $K_1$-group denoted by $K_1^w(R[G])$ for a ring $\Z \subseteq R \subseteq \C$, where the square $n$-matrices over $R[G]$ need not be invertible but weak isomorphisms instead, that is, injective and such that the associated right multiplication operator in $\ell^2(G)^n$ has dense image. This group is also interesting for the universal $L^2$-torsion introduced in \cite{FriedlLuck_l2torsion}. They showed that for torsion-free groups in Linnell's class $\mathcal{C}$, $K_1^w(R[G])$ can be completely understood in terms of $\mathcal{R}_{\scalebox{0.8}{$F[G]$}}$ where $F \subseteq \C$ is the field of fractions of $R$. The situation here is similar:

    \begin{thm} \label{thm: Linnell_Luck_localization}
        Let $G$ be a countable group in the class $\mathcal{T}_{\scalebox{0.7}{$\mathcal{VLI}$}}$ with finite $\lcm(G)$ and let $R$ be a ring with $\Z \subseteq R \subseteq \C$ closed under complex conjugation. Denote by $F \subseteq \C$ its field of fractions. Then $\mathcal{R}_{\scalebox{0.8}{$F[G]$}}$ is a semisimple ring, a universal localization of $F[G]$ and there are isomorphisms
        \[
        K_1^w(R[G]) \overset{\cong}{\longrightarrow} K_1(\mathcal{R}_{\scalebox{0.8}{$F[G]$}}) \overset{\cong}{\longrightarrow} \prod_{i=1}^l \mathcal{D}_i^{\times}/[\mathcal{D}_i^{\times}, \mathcal{D}_i^{\times}]
        \]
        where $\mathcal{D}_i$ are the division rings corresponding to the Artin--Wedderburn decomposition of $\mathcal{R}_{\scalebox{0.8}{$F[G]$}}$.
    \end{thm}

\subsection*{Organization of the paper}

    In \cref{sec: prelims}, we recall some notions that will appear throughout the paper and prove some preliminary results on universal localizations. In \cref{sec: amalg_free_prod}, we study the strong Atiyah conjecture for amalgamated free products over finite groups and show its inheritance behaviour as well as the fact that the $\ast$-regular closure is a universal localization. In \cref{sec: HNN_ext}, we establish the analog results in the context of HNN-extensions. The statements shown in these two former sections are geared towards proving in \cref{sec: GG} that the strong and algebraic Atiyah conjectures are closed under the graph of groups construction if the edge groups are finite; we also show that the $\ast$-regular closure is a universal localization. In \cref{sec: class_T}, we extend the properties of Linnell's class $\mathcal{C}$ to $\mathcal{T}_{\scalebox{0.7}{$\mathcal{VLI}$}}$. We finish the paper with \cref{sec: comments_questions} where we discuss possible generalizations of our results.

\subsection*{Acknowledgments} 
    
    The author is grateful to Andrei Jaikin-Zapirain for many helpful conversations and for having suggested him the finite rank operators point of view towards the strong Atiyah conjecture. I am grateful to an anonymous referee, whose comments and corrections have much improved the article. The author is supported by \seqsplit{PID2020-114032GB-I00/AEI/10.13039/501100011033} of the Ministry of Science and Innovation of Spain.

\section{Preliminaries} \label{sec: prelims}

    Throughout, groups are countable, rings are assumed to be associative and unital, ring homomorphisms preserve the unit and modules are left modules unless otherwise specified.

\subsection{Von Neumann regular and \texorpdfstring{$\ast$}--regular rings}

    A ring $\UO$ is called {\it von Neumann regular} if for every $x\in \UO$ there exists $y\in \UO$ such that $xyx =x$. From a homological point of view, these rings possess a nice feature, namely all left (and right) modules are flat (cf. \cite[Corollary 1.13]{GoodearlvNr}). Furthermore:

    \begin{prop}{\cite[Theorem 1.11 \& Proposition 2.6]{GoodearlvNr}} \label{prop: vNr_proj_mod}
        Let $\UO$ be a von Neumann regular ring and $M$ a finitely presented left $\UO$-module. Then $M$ is projective and it admits a direct sum decomposition into cyclic projectives $\UO$-modules.
    \end{prop}

    The von Neumann regular rings we shall be working with have an additional structure, they are $\ast$-regular rings. A {\it $\ast$-ring} is a ring $R$ with a map $\ast: R \rightarrow R$ that is an involution, i.e. $(x^{\ast})^{\ast} = x, (x+y)^{\ast} = x^{\ast} + y^{\ast}$ and $(xy)^{\ast} = y^{\ast} x^{\ast}$ for $x,y\in R$. The involution is called {\it proper} if $x^{\ast} x = 0$ implies $x = 0$. A {\it $\ast$-regular ring} is a von Neumann regular ring with a proper involution.

    If $\UO$ is a $\ast$-regular ring, then for each $x \in \UO$ there is a unique element $y \in \UO$, called the \textit{relative inverse of} $x$, such that $xyx = x$, $yxy = y$ and both $xy$ and $yx$ are projections, that is idempotents and self-adjoints; see \cite[Proposition 51.4(i)]{Berberian72}. For a $\ast$-subring $R$ of a $\ast$-regular ring $\mathcal{U}$, P. Ara and K. R. Goodearl realized that one can construct the smallest $\ast$-regular subring of $\mathcal{U}$ containing $R$, termed the \textit{$\ast$-regular closure} and denoted by $\mathcal{R}(R, \mathcal{U})$, in the same spirit of the division closure of a subring of a ring. The construction in \cite[Proposition 6.2]{AraGoodearlRegClosure} goes as follows:
    \begin{enumerate}[label=(\roman*)]
        \item Set $\mathcal{R}_0(R, \mathcal{U}) := R$, a $\ast$-subring of $\UO$.
        \item Suppose $n \geq 0$ and that we have constructed a $\ast$-subring $\mathcal{R}_n(R, \UO)$ of $\UO$. Let $\mathcal{R}_{n+1}(R, \UO)$ be the subring of $\UO$ generated by the elements of $\mathcal{R}_n(R, \UO)$ and their relative inverses in $\UO$, which is a $\ast$-subring.
    \end{enumerate}
    Then $\mathcal{R}(R, \UO) = \cup_{n = 0}^{\infty} \mathcal{R}_n(R, \UO)$.

    An \textit{$R$-ring} is a pair $(S, \varphi)$ where $\varphi : R \rightarrow S$ is a ring homomorphism. We will often omit $\varphi$ if it is clear from the context.

\subsection{Sylvester rank functions}\label{subsec:slyv}

    Let $R$ be a ring. A \textit{Sylvester matrix rank function} $\rk$ on $R$ is a function that assigns a non-negative real number to each matrix over $R$ and satisfies the following conditions:
    \begin{enumerate}[label=(\roman*)]
        \item $\rk(A)=0$ if $A$ is any zero matrix and $\rk(1)=1$;
        \item $\rk(AB)\leq \min\{\rk(A),\rk(B)\}$ for any matrices $A$ and $B$ which can be multiplied;
        \item $\rk\begin{psmallmatrix}
             A & 0 \\
             0 & B
        \end{psmallmatrix}=\rk(A)+\rk(B)$ for any matrices $A$ and $B$; and
        \item $\rk \begin{psmallmatrix}
             A & C \\
             0 & B
        \end{psmallmatrix} \geq \rk(A)+\rk(B)$ for any matrices $A$, $B$ and $C$ of appropriate sizes.
    \end{enumerate}
    
    We denote by $\mathbb{P}(R)$ the set of Sylvester matrix rank functions on $R$. Note that a ring homomorphism $\varphi\colon R\rightarrow S$ induces a map $\varphi^{\#}\colon \mathbb{P}(S)\rightarrow \mathbb{P}(R)$. Indeed, we can pull back any rank function $\rk$ on $S$ to a rank function $\varphi^{\#}(\rk)$ on $R$ by setting
    \[
    \varphi^{\#}(\rk)(A):=\rk(\varphi(A))
    \]
    for every matrix $A$ over $R$. We shall often abuse notation and write $\rk$ instead of $\varphi^{\#}(\rk)$ when it is clear that we are referring to the rank function on $R$. There is a dual notion of Sylvester module rank function. 
    
    \begin{defn}\label{def: SMRF}
        A {\it Sylvester module rank function}, or simply {\it $\SMRF$}, $\dim$ on $R$ is a map that assigns a non-negative real number to each finitely presented left $R$-module and satisfies the following properties:
    \begin{enumerate}[label=(\roman*)]
        \item $\dim(0)=0$ and $\dim(R)=1$;
        \item $\dim(M_1\oplus M_2)=\dim(M_1)+\dim(M_2)$; and
        \item if $M_1\rightarrow M_2\rightarrow M_3\rightarrow 0$ is exact, then
        \[
        \dim(M_1)+\dim(M_3)\geq \dim(M_2)\geq \dim(M_3). 
        \]
    \end{enumerate}
    \end{defn}
    
    A Sylvester module rank function is called \textit{exact} if $\dim(M_1) + \dim(M_3) = \dim(M_2)$ whenever $0 \rightarrow M_1 \rightarrow M_2 \rightarrow M_3 \rightarrow 0$ is exact. Note that by \cref{prop: vNr_proj_mod} this is always the case for von Neumann regular rings.

    As one might expect, there is a bijective correspondence between these two notions defined as follows. Given a Sylvester matrix rank function $\rk$ on $R$, we can define an $\SMRF$ $\dim$ on $R$ by assigning to any finitely presented left $R$-module with presentation $M=R^m/R^nA$ for some $A\in \Mat_{n\times m}(R)$, the value $\dim(M):=m-\rk(A)$; conversely, if $\dim$ is a $\SMRF$ on $R$, then we obtain a Sylvester matrix rank function setting $\rk(A):=m-\dim(R^m/R^nA)$ (cf. \cite[Theorem 4]{Malcolmson} and \cite[Proposition 1.2.8]{DLopezAlvarezThesis}).

\subsection{The von Neumann dimension and the \texorpdfstring{$\ast$}--regular closure of a group}\label{subsec: vNdim_reg_clos}

    A countable group $G$ acts by left and right multiplication on $\ell^2(G)$. We can extend the right action of $G$ on $\ell^2(G)$ to the group algebra $\C[G]$, and hence consider $\C[G]$ as a subalgebra of $\mathcal{B}(\ell^2(G))$, namely the {\it bounded linear operators on} $\ell^2(G)$.

    A finitely generated {\it Hilbert} $G$-module is a closed subspace $V\leq \ell^2(G)^n$ invariant under the left action of $G$. Given such $V$, there exists $\pr_V$, the orthogonal projection onto $V$. We set 
    \[
    \dim_G(V):=\trace_G(\pr_V):=\sum_{i=1}^n \langle \pr_V(1_i),1_i\rangle_,
    \]
    where $1_i$ is the element of $\ell^2(G)^n$ having $1$ in the $i$th entry and $0$ in the rest of the entries. We call this number the {\it von Neumann dimension} of $V$. This dimension induces a Sylvester matrix rank function on $\C[G]$. Concretely, given $A\in \Mat_{n\times m}(\C[G])$ we can then associate the natural bounded linear operator $\phi_G^A:\ell^2(G)^n\rightarrow\ell^2(G)^m$ given by right multiplication, and define $\rk_G(A)=\dim_G \overline{\im \phi_G^A}$ where the bar stands for the closure of the image subspace.

    The {\it group von Neumann algebra} $\VN(G)$ of $G$ is the algebra of left $G$-equivariant bounded operators on $\ell^2(G)$
    \[
    \VN(G):=\{\phi\in \mathcal{B}(\ell^2(G)): \phi(g\cdot v) = g \cdot \phi(v) \mbox{ for all }g\in G,v\in \ell^2(G)\}.
    \]
    The ring $\VN(G)$ is a finite von Neumann algebra. Moreover, $\VN(G)$ satisfies the left (and right) Ore condition with respect to the set of non-zero-divisors (a result proved by S. K. Berberian in \cite{Berberian82}). We denote by $\UO(G)$ the left (resp. right) classical ring of fractions and called it the {\it ring of unbounded operators affiliated to} $G$. The reader can consult the Ph.D thesis of H. Reich \cite{ReichThesis} for the following standard facts. The ring $\UO(G)$ can be defined alternatively as the set of densely defined (unbounded) operators which are closed, i.e. its graph is closed in $\ell^2(G)^2$, and commute with the left action of $G$. It is a $\ast$-regular ring where the involution is given by taking the adjoint operator. Besides, observe that every $x\in \ell^2(G)$ defines a map from $\C[G]$ to $\ell^2 (G)$ given by right multiplication. This map is densely defined, commutes with the left action, and moreover, it can be shown to extend to a closed left $G$-equivariant operator. Thus we have the inclusion $\ell^2(G) \subseteq \UO(G)$. Not only that, this embedding extends the embedding $\VN(G) \subseteq \UO(G)$, because the map induced by right multiplication by $\phi(e)$ coincides with $\phi$ in $\C[G]$, and hence $\phi$ is its unique closed extension in $\UO(G)$. To sum up, we have the following chain of $\ast$-embeddings
    \[
    \C[G] \subseteq \VN(G) \subseteq \ell^2(G) \subseteq \UO(G).
    \]
    There is another chain of $\ast$-embeddings produced by the $\ast$-regular closure. Specifically, since $\UO(G)$ is a $\ast$-regular ring, given $K \subseteq \C$ a field closed under complex conjugation, we set $\mathcal{R}_{\scalebox{0.8}{$K[G]$}}$ to be the $\ast$-regular closure of $K[G]$ in $\mathcal{U}(G)$, or in other words, $\mathcal{R}_{\scalebox{0.8}{$K[G]$}}:= \mathcal{R}(K[G], \UO(G))$. Then 
    \[
    K[G] \subseteq \mathcal{R}_{\scalebox{0.8}{$K[G]$}} \subseteq \UO(G)
    \]
    is another such chain. 
    
    The Sylvester matrix rank function $\rk_G$ can be extended to $\UO(G)$ as follows
    \[
    \rk_G(s^{-1}r):=\rk_G(r)=\dim_G \left( \overline{r \ell^2(G)} \right).
    \]
    This Sylvester matrix rank function is indeed the one associated to the $L^2$-dimension $\dim_{\scalebox{0.8}{$\UO(G)$}}$ \cite[Definition 8.28]{Luck02} (cf. \cite[Lemma 4.2.4]{DLopezAlvarezThesis}). We record a well-known result that can be found in \cite[Proposition 4.2.2]{DLopezAlvarezThesis}.

    \begin{prop}\label{prop: preserve_dim_gps}
        For any subgroup $H$ of $G$ it holds that
        \[
        \dim_{\scalebox{0.8}{$\UO(H)$}}(M)=\dim_{\scalebox{0.8}{$\UO(G)$}}(\UO(G)\otimes_{\scalebox{0.8}{$\UO(H)$}}M)
        \]
        for every left $\UO(H)$-module $M$.
    \end{prop}

    The inclusion $\mathcal{R}_{\scalebox{0.8}{$K[G]$}} \subseteq \UO(G)$ induces an $\SMRF$ that is faithful over projective modules.

    \begin{prop}\label{prop: dimG_proj_faithful}
        Let $P$ be a finitely generated projective left $\mathcal{R}_{\scalebox{0.8}{$K[G]$}}$-module such that $\dim_{\scalebox{0.8}{$\UO(G)$}} \UO(G) \otimes_{\scalebox{0.8}{$\mathcal{R}_{\scalebox{0.8}{$K[G]$}}$}} P$ is zero. Then $P = \{0\}$.
    \end{prop}
    \begin{proof}
        Since $P$ must be finitely presented, we deduce from \cref{prop: vNr_proj_mod} that there exist $x_1, \ldots , x_n \in \mathcal{R}_{\scalebox{0.8}{$K[G]$}}$ such that $P \cong \mathcal{R}_{\scalebox{0.8}{$K[G]$}}x_1 \oplus \ldots \oplus \mathcal{R}_{\scalebox{0.8}{$K[G]$}}x_n$. Recall that every right $\mathcal{R}_{\scalebox{0.8}{$K[G]$}}$-module is flat being $\mathcal{R}_{\scalebox{0.8}{$K[G]$}}$ von Neumann regular. So the map $\UO(G) \otimes_{\scalebox{0.8}{$\mathcal{R}_{\scalebox{0.8}{$K[G]$}}$}} \mathcal{R}_{\scalebox{0.8}{$K[G]$}}x_{i} \rightarrow \UO(G)$ induced from the injective map $\mathcal{R}_{\scalebox{0.8}{$K[G]$}}x_{i} \hookrightarrow \mathcal{R}_{\scalebox{0.8}{$K[G]$}}$ is again injective. Thus for each $i \in \{1, \ldots , n\}$ it holds that 
        \[
        0 = \dim_{\scalebox{0.8}{$\UO(G)$}} \UO(G) \otimes_{\scalebox{0.8}{$\mathcal{R}_{\scalebox{0.8}{$K[G]$}}$}} \mathcal{R}_{\scalebox{0.8}{$K[G]$}}x_{i} = \dim_{\scalebox{0.8}{$\UO(G)$}} \UO(G)x_i = \rk_G(x_i).
        \]
        However, $\rk_G$ is faithful over $\UO(G)$; see, for instance, \cite[Proposition 4.2.2]{DLopezAlvarezThesis}. So $x_i = 0$ for $i = 1, \ldots, n$, and hence $P = \{0\}$.
    \end{proof}

    Let $f:R\rightarrow S$ be a ring homomorphism. We say that $f$ is {\it epic} if for every ring $T$ and every pair of homomorphisms $\alpha, \beta:S\rightarrow T$, the equality $\alpha \circ f = \beta \circ f$ implies $\alpha = \beta$. Our interest in $\mathcal{R}_{\scalebox{0.8}{$K[G]$}}$ rests on the fact that the inclusion $K[G] \subseteq \mathcal{R}_{\scalebox{0.8}{$K[G]$}}$ is epic according to \cite[Proposition 6.1]{JaikinBaseChange}. In a sense, the $\ast$-regular ring $\mathcal{R}_{\scalebox{0.8}{$K[G]$}}$ is not too big for $K[G]$, unlike $\UO(G)$. This property will be crucial for us due to the following characterization of epic homomorphisms.
    
    \begin{prop}{\cite[Proposition XI.1.2]{Stenstrom_RingsofQuot}}\label{prop: epic_tensor_prod}
        Let $f:R\rightarrow S$ be a ring homomorphism. Then $f$ is epic if and only if the multiplicative map
        \[
        m:S\otimes_R S\rightarrow S
        \]
        is an isomorphism of $S$-bimodules.
    \end{prop}

\subsection{The strong and algebraic Atiyah conjectures}

    For a group $G$, we set $\lcm(G)$ as the least common multiple of the set $\{|F|: F \leqslant G, |F| < \infty \}$ if it is finite, and $\infty$ otherwise.

    \begin{defn}
        Let $G$ be a group with finite $\lcm(G)$ and $K\subseteq \C$ a field. We say that the {\it strong Atiyah conjecture for $G$ holds over $K$} if every finitely presented left $K[G]$-module $M$ satisfies
        \[
        \dim_{\scalebox{0.8}{$\UO(G)$}}(\UO(G) \otimes_{\scalebox{0.8}{$K[G]$}} M) \in \frac{1}{\lcm(G)}\Z.
        \]
    \end{defn}

    Whenever the subfield $K$ is closed under complex conjugation, there is a reformulation of the strong Atiyah conjecture, given in \cite[Proposition 2.4.10]{FabianHennekeThesis}, that is more convenient for us.

    \begin{prop}\label{prop: Atiyah_conj_reg_closure}
        Let $G$ be a group with fintie $\lcm(G)$ and $K\subseteq \C$ a field closed under complex conjugation. Then the strong Atiyah conjecture for $G$ holds over $K$ if and only if every (arbitrary) left $\mathcal{R}_{\scalebox{0.8}{$K[G]$}}$-module $M$ satisfies
        \[
        \dim_{\scalebox{0.8}{$\UO(G)$}}(\UO(G)\otimes_{\scalebox{0.8}{$\mathcal{R}_{\scalebox{0.8}{$K[G]$}}$}} M) \in \frac{1}{\lcm(G)}\Z \cup \{\infty\}.
        \]
    \end{prop}

    Recall that a ring $R$ is called {\it semisimple} if every $R$-module is projective. It is not necessary to distinguish between left or right $R$-modules because the Artin--Wedderburn theorem states that $R$ is isomorphic to a finite direct product of full matrix rings over division rings. We record a structural consequence of the strong Atiyah conjecture proved in \cite[Proposition 2.4.6]{FabianHennekeThesis}.

    \begin{prop}\label{prop: Atiyah_implies_semisimple}
        Let $G$ be a group with finite $\lcm(G)$ and $K \subseteq \C$ a field closed under complex conjugation. If the strong Atiyah conjecture for $G$ holds over $K$, then the ring $\mathcal{R}_{\scalebox{0.8}{$K[G]$}}$ is semisimple.
    \end{prop}

    For a ring $R$, we recall that $K_0(R)$ is the free abelian group on the set of finitely generated projective left $R$-modules modulo the relation $[P_1] + [P_3] = [P_2]$ if there is a short exact sequence
    \[
    0 \rightarrow P_1 \rightarrow P_2 \rightarrow P_3 \rightarrow 0.
    \]
    Every ring homomorphism $f : R \rightarrow S$ induces a map from $K_0(R)$ to $K_0(S)$ that sends $[P]$ to $[S \otimes_R P]$.
    
    \begin{defn}{\cite[Conjecture 6.2]{Jaikin_l2survey}}
        Let $G$ be a group with finite $\lcm(G)$ and $K\subseteq \C$ a field closed under complex conjugation. We say that the {\it algebraic Atiyah conjecture for $G$ holds over $K$} if the map 
        \[
        \bigoplus_{F \leqslant G, |F|<\infty} K_0(K[F]) \rightarrow K_0(\mathcal{R}_{\scalebox{0.8}{$K[G]$}})
        \]
        is surjective. We call this map the {\it algebraic Atiyah map}.
    \end{defn}

    It follows from \cref{prop: vNr_proj_mod} that the algebraic Atiyah conjecture implies the strong Atiyah conjecture. More specifically, every finitely presented $\mathcal{R}_{\scalebox{0.8}{$K[G]$}}$-module $M$ is projective. So, if the algebraic Atiyah conjecture is satisfied, there exist $P_1$ and $P_2$ induced modules coming from projective $K[F]$-modules for some collection of finite subgroups $F$ of $G$ satisfying $M \oplus P_1 \cong P_2$. But then, according to \cref{prop: preserve_dim_gps}, the $L^2$-dimension of each induced module belongs to $\frac{1}{\lcm(G)} \Z$, and therefore, so does the $L^2$-dimension of $M$.

\subsection{Finite rank operators}\label{subsec: finite_rank_op}

    A {\it finite rank operator} is a bounded linear operator between Banach spaces whose range is finite dimensional. For a positive integer $n$, an operator $T$ on a Hilbert space $\mathcal{H}$ of finite rank $n$ takes the form
    \[
    T(h) = \sum_{i=1}^n \langle h, u_i \rangle v_i  
    \]
    where $\{u_i : 1 \leq i \leq n\}$ are vectors in $\mathcal{H}$ and $\{v_i : 1 \leq i \leq n\}$ is an orthonormal basis in $T(\mathcal{H})$, and $h\in \mathcal{H}$. Thus, for two vectors $u, v \in \mathcal{H}$, $T_{u,v}$ will stand for the finite rank operator defined as $T_{u,v}(h) := \langle h, u \rangle v$ for every $h \in \mathcal{H}$. We denote the class of finite rank operators over a Hilbert space $\mathcal{H}$ by $F(\mathcal{H})$. Note that it is a left and right ideal in $\mathcal{B}(\mathcal{H})$.

    For this class of operators a trace functional can be defined. Concretely, if $\{e_i : i \in I\}$ is a basis for $\mathcal{H}$, define $\trace: F(\mathcal{H}) \rightarrow \C$ by 
    \[
    \trace(T) = \sum_{i \in I} \langle T e_i, e_i \rangle.
    \]
    A few comments are in order. First, the above sum is absolutely convergent. Indeed, for $u, v \in \mathcal{H}$, we have that $u = \sum_{i \in I} \alpha_i e_i$ and $v = \sum_{i \in I} \beta_i e_i$. Thus, if we set $u' = \sum_{i \in I} |\alpha_i| e_i$ and $v' = \sum_{i \in I} |\beta_i| e_i$, both of which are elements in $\mathcal{H}$, it holds that
    \[
        \sum_{i\in I} |\langle T_{u,v}(e_i), e_i \rangle | = \sum_{i\in I} |v_{i} \overline{u_{i}}| = \langle v', u' \rangle < \infty.
    \]
    Second, because of the absolute convergence, it is a straightforward computation to check that the trace functional does not depend on the choice of the basis. Moreover, if $S,T\in F(\mathcal{H})$ then $\trace(ST) = \trace(TS)$ (see \cite[Exercise IX.2.20]{Conway_FunctionalAnal} for further properties).

\subsection{Graphs of rings} \label{subsec: graph_rings}

    Graphs of rings are defined in complete analogy with graphs of groups. The amalgamated product of rings over a common subring or the coproduct of rings was studied extensively by Bergman in \cite{Bergman_Coprod}, and the HNN extension of rings was introduced and studied by Dicks in \cite{Dicks_HNN}. We recall here the definitions:
    \begin{enumerate}[label=(\roman*)]
        \item Coproducts of rings: Let $R_e, R_{v_1}$ and $R_{v_2}$ be rings with ring homomorphisms $R_e \rightarrow R_{v_i}$ for $i=1,2$. We shall denote by $R_{v_1} \ast_{R_e} R_{v_2}$ the {\it coproduct of the $R_{v_i}$ over $R_e$} defined as the pushout of the maps $R_e \rightarrow R_{v_i}$ for $i=1,2$ in the category of rings. This construction will correspond to a graph with two vertices and a single edge.
        \item HNN of rings: Let $R_e$ and $R_v$ be rings with ring homomorphisms $\varphi_e, \varphi_{\overline{e}} : R_e \rightarrow R_v$. The {\it HNN extension of $R_v$ over $R_e$} is defined as an $R_v$-ring $(R , \eta)$ universal with respect to the existence of a distinguished unit $t_e$ such that the inner automorphism $i_{t_e} : R \rightarrow R, r \mapsto t_e^{-1} r t_e$ makes the diagram
        \[
        \begin{tikzcd}
            & R_v \arrow[r, "\eta"] & R \arrow[dd, "i_{t_e}"] \\
            R_e \arrow[ru, "\varphi_e"] \arrow[rd, "\varphi_{\overline{e}}"] & &\\
            & R_v \arrow[r, "\eta"] & R 
        \end{tikzcd}
        \]
        commutative. We denote it by $R_{v_{R_e}}\langle t, t^{-1} ; \varphi_{\overline{e}}\rangle$. This construction will correspond to a graph with a single vertex and one edge.
    \end{enumerate}
    It is worth mentioning that most of the results for modules over these ring constructions required that the ring homomorphisms are injective. In \cite[Section 3]{FisherPeralta_Kaplansky'sZD3mfld} both constructions are unified under the definition of graph of rings. Our graphs are connected and oriented, with $\overline{e}$ denoting the same edge as $e$ but with the opposite orientation. Every edge $e$ has an origin vertex $o(e)$ and a terminus vertex $t(e)$ such that $o(e) = t(\overline e)$. Graphs are allowed to have loops and multiple edges.

    \begin{defn}{\cite[Definition 3.1]{FisherPeralta_Kaplansky'sZD3mfld}}\label{def:treegraph}
        Let $\Gamma$ be a graph and let $T$ be a spanning tree. For each vertex $v$ of $\Gamma$ we have a \textit{vertex ring} $R_v$ and for each edge $e$ of $\Gamma$ we have an \textit{edge ring} $R_e$ and we impose $R_e = R_{\overline{e}}$ for every edge $e$. Moreover, for each (directed) edge $e$ there is an injective ring homomorphism $\varphi_e \colon R_e \rightarrow R_{t(e)}$. Then the \textit{graph of rings} $\mathscr R_{\Gamma, T} = (R_v, R_e)$ is the ring defined as follows:
        \begin{enumerate}[label=(\roman*)]
            \item For each edge of $e$ of $\Gamma$ we introduce a formal symbol $t_e$;
            \item $\mathscr R_{\Gamma, T}$ is generated by the vertex rings $R_v$ and the elements $t_e, t_e\inv$ and subjected to the relations
            \begin{itemize}
                \item $t_{\overline e}t_e = t_e t_{\overline e} = 1$;
                \item $t_e \varphi_{\overline{e}}(r) t_{\overline e} = \varphi_e(r)$ for all $r \in R_e$; and
                \item if $e \in T$, then $t_e = 1$. 
            \end{itemize}
        \end{enumerate}
    \end{defn}

    According to \cite[Proposition 3.4]{FisherPeralta_Kaplansky'sZD3mfld} the isomorphism type of $\mathscr R_{\Gamma,T}$ is independent of the choice of $T$. We will thus simplify the notation and denote the graph of rings by $\mathscr R_\Gamma$. Note that every graph of rings can be broken into two special ring constructions, namely coproducts of rings and HNN of rings. Most of the work done in \cite{Bergman_Coprod} and \cite{Dicks_HNN} is geared towards characterizing the structure of induced (also called standard by Bergman) modules and projective modules. By an {\it induced} (left) $\mathscr{R}_{\Gamma}$-module, we shall mean an $\mathscr{R}_{\Gamma}$-module of the form $M = \oplus_{v \in V} \mathscr{R}_{\Gamma} \otimes_{R_{v}} M_{v}$, where each $M_{v}$ is a left $R_{v}$-module. The following result is a direct consequence of \cite[Theorem 2.2]{Bergman_Coprod} and \cite[Theorem 21]{Dicks_HNN}.

    \begin{thm}\label{thm: induced_submod}
        Let $\mathscr R_{\Gamma} = (R_v, R_e)$ be a graph of rings with finite graph $\Gamma$. Assume that $R_e$ is a semisimple ring for every edge $e$. Then any left $\mathscr R_{\Gamma}$-submodule of an induced $\mathscr R_{\Gamma}$-module is isomorphic to an induced $\mathscr R_{\Gamma}$-module.
    \end{thm}

    Similarly, for projectives modules a combination of \cite[Corollary 2.6]{Bergman_Coprod} and \cite[Corollary 22]{Dicks_HNN} yields the next upshot.

    \begin{prop}\label{prop: proj_mod_GR}
        Let $\mathscr R_{\Gamma} = (R_v, R_e)$ be a graph of rings with finite graph $\Gamma$. Assume that $R_e$ is a semisimple ring for every edge $e$. Then every projective left $\mathscr R_{\Gamma}$-module $P$ has the form $\oplus_{v \in V} \mathscr{R}_{\Gamma} \otimes_{R_{v}} P_{v}$, where $P_v$ is a projective left $R_v$-module for every $v\in V$.
    \end{prop}

    There are also right module versions of the above results.

\subsection{Universal localization} \label{subsec: univ_local}

    Let $R$ be a ring and $\Sigma$ a set of maps between finitely generated projective left $R$-modules. Then there is an $R$-ring $R_{\Sigma}$ universal with respect to the property that every element $R_{\Sigma} \otimes_R \alpha$ for $\alpha \in \Sigma$ has an inverse cf. \cite[Theorem 4.1]{Schofield85}. We call $R_{\Sigma}$ a \textit{universal localization of $R$ at $\Sigma$}. One can give an explicit construction of $R_{\Sigma}$ following \cite[Construction 2.1]{BergmanDicks_UnivLocal}. To invert a map $\alpha: P \rightarrow Q$ in $\Sigma$ between finitely generated projective left $R$-modules, first choose idempotent endomorphisms $e,f$ of a free $R$-module $R^n$ with images $P$ and $Q$, respectively. Then $a = \alpha e$ is an endomorphism of $R^n$ with $ae = a = fa$. A map from $Q$ to $P$ inverse of $a$ would be represented by a matrix $a'$ satisfying
    \[
    e a' = a' = a' f, a' a = e \mbox{ and } a a' = f.
    \]
    The above equations are matrix equations, so we adjoin indeterminates $a_{ij}$ to $R$ and impose the above relations.

    Let $\dim$ be an $\SMRF$ over $R$. Given a map $\alpha : P \rightarrow Q$ between two finitely generated projective left $R$-modules, define $\dim(\alpha)$ to be the infimum over all numbers $\dim(P')$, where $P'$ ranges over all finitely generated projective left $R$-modules $P'$ for which there exists a factorization $\alpha: P \rightarrow P' \rightarrow Q$. We call $\alpha$ a {\it full map with respect to $\dim$} if $\dim(P) = \dim(\alpha) = \dim(Q)$. We denote by $R_{\dim}$ the universal localization of $R$ at the set of all full maps with respect to $\dim$.

    \begin{lem}\label{lem: invertiv_is_full}
        Let $R$ be a ring, $R_{\Sigma}$ a universal localization of $R$ and $\dim$ an $\SMRF$ on $R_{\Sigma}$. Every map between finitely generated projective left $R$-modules that becomes invertible over $R_{\Sigma}$ is full with respect to the induced dimension on $R$.
    \end{lem}
    \begin{proof}
        In this proof, we will abuse notation and write $\dim(M)$ for the induced dimension of a finitely presented left $R$-module $M$ instead of $\dim(R_{\Sigma} \otimes_R M)$. Suppose that $\alpha: P \rightarrow Q$ is a map of finitely generated projective left $R$-modules that becomes invertible over $R_{\Sigma}$. Then $R_{\Sigma} \otimes_ R P \cong R_{\Sigma} \otimes_R Q$, and in particular, $\dim(P) = \dim(Q)$. Let $P'$ be a finitely generated projective left $R$-module for which there exists a factorization
        \[
        \alpha : P \xrightarrow{\alpha_1} P' \xrightarrow{\alpha_2} Q.
        \]
        Then $\coker(\alpha)$ surjects onto $\coker(\alpha_2)$. Since taking tensor product is right exact and $R_{\Sigma} \otimes_R \coker(\alpha)$ vanishes, $R_{\Sigma} \otimes_R \coker(\alpha_2)$ also vanishes. Hence, we get that
        \[
            \dim(P') = \dim(P') + \dim(\coker(\alpha_2))  \geq \dim(Q).
        \]
        Thus we conclude that $\dim(\alpha) = \dim(Q)$, which shows that $\alpha$ is full.
    \end{proof}
    
    In some sense, full maps are the maps that have a chance of becoming invertible. Actually, over a von Neumann regular ring the following holds due to \cref{prop: vNr_proj_mod}.

    \begin{lem}\label{lem: full_vNr}
        Let $R$ be a von Neumann regular ring and let $\dim$ be an $\SMRF$ faithful over finitely generated projective left $R$-modules. Then every full map over $\dim$ is an isomorphism.
    \end{lem}
    \begin{proof}
        Let $\alpha : P \rightarrow Q$ be a full map between two finitely generated projective left $R$-modules. Note that $\coker(\alpha)$ is a finitely presented $R$-module. Thus, by \cref{prop: vNr_proj_mod} $\coker(\alpha)$ is projective as $R$-module, and hence, $Q$ splits as a direct sum of $\im(\alpha)$ and $\coker(\alpha)$. In particular, $\im(\alpha)$ is a finitely generated projective $R$-module. But, recall that $\alpha$ is a full map, so $\dim(\im(\alpha))= \dim(Q)$. In other words, $\coker(\alpha)$ is a projective $R$-module of zero dimension. Since the $\SMRF$ is faithful over projective $R$-modules, $\coker(\alpha)$ must vanish, and therefore $\alpha$ is surjective. To show that $\alpha$ is also injective, note that $P$ decomposes as a direct sum of $\ker(\alpha)$ and $Q$ being $Q$ a projective $R$-module. So, $\ker(\alpha)$ is another projective $R$-module of zero dimension, that is, the zero module. This shows that $\alpha$ is injective, and hence an isomorphism.
    \end{proof}
    
    A left $R_{\Sigma}$-module $M$ is called \textit{stably induced} if there exists some left $R$-module $N$ and a non-negative integer $n$ such that $M \oplus R_{\Sigma}^n \cong R_{\Sigma} \otimes_R N$. It turns out that under the following conditions every finitely generated projective $R_{\Sigma}$-module is of this form. We recall that a ring $R$ is called \textit{hereditary} if all of its left and right ideals are projective $R$-modules.

    \begin{prop}\label{prop: project_mod_localization}
        Let $R$ be a hereditary ring with an SMRF $\dim$ that is faithful over finitely generated projective left $R$-modules and let $\Sigma$ be a collection of full maps with respect to $\dim$. Then any finitely generated projective left $R_{\Sigma}$-module $Q$ satisfies an equation of the form
        \[
        Q \oplus R_{\Sigma}^n \cong R_{\Sigma} \otimes_R P
        \]
        where $P$ is a finitely generated projective left $R$-module.
    \end{prop}
    \begin{proof}
        This is a particular case of \cite[Theorem 5.2]{Schofield85} and the proof goes as in \cite[Theorem 10.74]{Luck02}.
    \end{proof}

    Another interesting feature of localization is its transitivity whenever the localization is done over stably induced modules.

    \begin{prop}{\cite[Theorem 4.6]{Schofield85}}\label{prop: local_transitive}
        Let $\Sigma$ be a collection of maps between finitely generated projectives over $R$, and let $\Sigma'$ be a collection of maps between stably induced finitely generated projectives over $R_{\Sigma}$. Then $(R_{\Sigma})_{\Sigma'}$ is a universal localization of $R$ at a suitable set of maps between finitely generated projectives over $R$ that contains $\Sigma$.
    \end{prop}
    
    Surprisingly, Schofield developed a computational criterion to recognize universal localizations. The following statement is a slight variation of \cite[Theorem 12.3]{Schofield85} which is more convenient for our purposes.

    \begin{thm}\label{thm: recogn_universal_localization}
        Let $R$ be a left hereditary ring and let $\phi: R \rightarrow S$ be an epic homomorphism from $R$ to a semisimple ring $S$ with $\SMRF$ $\dim$. Assume that $\phi^{\#}(\dim)$ is faithful over the finitely generated projective left $R$-modules. If $\Tor_1^R(S,S) = 0$, then $S = R_{\dim}$.
    \end{thm}
    \begin{proof}
        From the proof in \cite[Theorem 12.3]{Schofield85} (see also \cite[Theorem 5.3]{Schofield85}), we know that $S$ is the universal localization of $R_{\dim}$ at some idempotent element $e$ and that $\phi^{\#}(\dim)$ extends to a faithful $\SMRF$ $\widetilde{\dim}$ over the finitely generated projective left $R_{\dim}$-modules. Let $Q$ denote the left ideal of $R_{\dim}$ generated by $1-e$. Then, according to \cref{prop: project_mod_localization}, there exists some finitely generated projective left $R$-module $P$ and some non-negative integer $n$ such that
        \[
        Q \oplus R_{\dim}^n \cong R_{\dim} \otimes_R P.
        \]
        But now, if we extend scalars to $S = (R_{\dim})_{\{e\}}$, we get that $S \otimes_R P \cong S^n$. In particular, $\phi^{\#}(\dim)(P) = n$, and hence $\widetilde{\dim}(Q) = 0$. However, $\widetilde{\dim}$ is faithful over finitely generated projective modules. Therefore, $Q=0$ and $e = 1$ as we wanted.
    \end{proof}

\section{Amalgamated free products over finite groups} \label{sec: amalg_free_prod}

    In this section we prove that the strong Atiyah conjecture is closed under amalgamated free products over finite groups. Then, we move on to show the algebraic conjecture using universal localizations. For our first goal, we take inspiration from the work of Mai, Speicher and Yin \cite{MaiSpeicherYinAtiyahFree}, where finite rank operators are used to show the Atiyah conjecture for free groups, and an algebraic reinterpretation of its article done by Jaikin-Zapirain in \cite{Jaikin_ExplicitUFree}. Before going into the proofs, we establish a crucial result for finite rank operators on $\ell^2(G)$. Note that $F(\ell^2(G))$ possesses a right $\C[G]$-module structure given by
    \[
    T\cdot g := R_{g} \circ T\circ R_{g^{-1}}
    \]
    for $T\in F(\ell^2(G))$ and $g\in G$, where $R_g$ stands for the bounded operator defined as right multiplication with $g$. Similarly, we will use $L_g$ for the left multiplication operator.

    Assume that $G$ is the amalgamated free product of $A$ and $B$ over the group $C$, that is, $G = A \ast_C B$. Recall that choosing left transversals for $C$ in $A$ and $B$, both containing $1$, we can write the elements of $G$ in its normal form; we fix such transversals. More concretely, we will write an element $g \in G$ as $g_n \cdots  g_2 g_1$ for $n \geq 1$ where $g_1 \in C$, $g_i$ is an element of either transversal for $C$ in $A$ or $B$ for $i > 1$, and $g_i,g_{i+1}$ belong to different transversals for $i>1$. We shall say that $g$ is of {\it length $n$} and that $g$ {\it ends in} $g_n$.
    
    Now, pick an element $\widetilde{a} \neq 1$ in the left transversal for $C$ in $A$. We are going to construct an operator $D_{\widetilde{a}} \in \mathcal{B}(\ell^2(G))$ that commutes with $R_b$ for all $b \in B$ and whose commutator with $R_a$ where $a \neq 1$ is in the left transversal for $C$ in $A$ equals to the commutator of $R_a$ and $\pr_{C} L_{\widetilde{a}^{-1}}$, where $\pr_{C}$ is just the projection onto $\ell^2(C)$. These conditions will become apparent in the proof of \cref{lem: amalg_prod_intersection_augment_ideal}. We define $D_{\widetilde{a}}$ on the natural basis of $\ell^2(G)$: given $g\in G$, $D_{\widetilde{a}}$ sends $g$ to $\widetilde{a}^{-1}g$ if the normal form of $g$ ends in $\widetilde{a}$, and to $0$ otherwise. We state a generalization of \cite[Proposition 6.3]{MaiSpeicherYinAtiyahFree} in our setting.

    \begin{lem}\label{lem: proj_operator_amal_product}
        With the above notation, the following holds:
        \begin{enumerate}[label=(\roman*)]
            \item $D_{\widetilde{a}} - R_{b^{-1}} D_{\widetilde{a}} R_{b} = 0$, for $b \in B$; and
            \item $D_{\widetilde{a}} - R_{a^{-1}} D_{\widetilde{a}} R_{a} = \pr_{C} L_{\widetilde{a}^{-1}} - R_{a^{-1}} \pr_{C} L_{\widetilde{a}^{-1}} R_{a}$, for $a \neq 1$ in the left transversal for $C$ in $A$.
        \end{enumerate}
    \end{lem}
    \begin{proof}
        To show these equalities of bounded operators, it suffices to know how they act on the basis of $\ell^2(G)$ given by the elements of the group. For {\it (i)} note that $D_{\widetilde{a}} R_{b} = R_{b} D_{\widetilde{a}}$. Indeed, if the normal form of $g$ ends in $b \in B$, then the normal form of $gb$ also ends in some element in $B$. On the other hand, whenever the normal form of $g$ ends in an element of the left transversal for $C$ in $A$ different from $1$, then the normal form of $gb$ ends in the same element, and hence $D_{\widetilde{a}} R_{b}(g)=R_{b} D_{\widetilde{a}}(g)$.

        For {\it (ii)} observe that if the normal form of $g$ is of length at least $3$ or if it ends in an element of the left transversal for $C$ in $B$ different from $1$, then the normal form of $ga$ ends in the same element, and hence $(D_{\widetilde{a}} -R_{a^{-1}} D_{\widetilde{a}} R_{a})(g) = 0$. But the right-hand side operator also vanishes at $g$, since both $g$ and $ga$ are either of length at least $3$ or end in an element of $B$, and $\pr_C$ is non-zero only for elements of $C$, which are of length $1$. Thus, it rests to consider $g \in A$. 
        
        First, $D_{\widetilde{a}}(g) \neq 0$ if and only if $g = \widetilde{a} c$ for some $c\in C$. But in this case, $R_{a^{-1}} D_{\widetilde{a}} R_{a}(g) = 0$ because $g a\neq \widetilde{a} c'$ for all $c'\in C$ since otherwise $a$ would belong to $C$, which is a contradiction. Second, $R_{a^{-1}} D_{\widetilde{a}} R_{a}(g) \neq 0$ if and only if $g = \widetilde{a}ca^{-1}$ for some $c\in C$, and in this situation $D_{\widetilde{a}}(g) = 0$ since $a$ is not in $C$. The same argument applies to the right-hand side operator. Finally, observe that
        \[
        D_{\widetilde{a}}(\widetilde{a}c) = c = \pr_{C} L_{\widetilde{a}^{-1}}(\widetilde{a}c)
        \]
        and
        \[
        R_{a^{-1}} D_{\widetilde{a}} R_{a}(\widetilde{a}ca^{-1}) = ca^{-1} = R_{a^{-1}} \pr_{C} L_{\widetilde{a}^{-1}} R_{a}(\widetilde{a}ca^{-1}).
        \]
    \end{proof}

    We are ready to show our key lemma. The proof arises from the one given in \cite[Proposition 6.2]{MaiSpeicherYinAtiyahFree}. We denote by $I_{\scalebox{0.8}{$\C[G]$}}$ the \textit{augmentation ideal} of the group algebra $\C[G]$. Recall that action of $G$ on $F(\ell^2(G))$ is given by $T\cdot g := R_{g} \circ T\circ R_{g^{-1}}$ for $T\in F(\ell^2(G))$ and $g\in G$.

    \begin{lem}\label{lem: amalg_prod_intersection_augment_ideal}
        Let $A$ and $B$ be two groups, $C$ a finite group and let $G=A \ast_C B$ be their amalgamated free product. Then the intersection of the $\C$-vector spaces $F(\ell^2(G))\cdot I_{\scalebox{0.8}{$\C[A]$}}$ and $F(\ell^2(G))\cdot I_{\scalebox{0.8}{$\C[B]$}}$ inside $F(\ell^2(G))$ is precisely $F(\ell^2(G))\cdot I_{\scalebox{0.8}{$\C[C]$}}$.
    \end{lem}
    \begin{proof}
        The inclusion $F(\ell^2(G))\cdot I_{\scalebox{0.8}{$\C[C]$}} \subseteq F(\ell^2(G))\cdot I_{\scalebox{0.8}{$\C[A]$}} \cap F(\ell^2(G))\cdot I_{\scalebox{0.8}{$\C[B]$}}$ is obvious. For the other direction, let $T\in F(\ell^2(G))\cdot I_{\scalebox{0.8}{$\C[A]$}} \cap F(\ell^2(G))\cdot I_{\scalebox{0.8}{$\C[B]$}}$. Then $T = \sum_{j=1}^n T_{j,A} \cdot (1-a_j)$ and $T = \sum_{i=1}^l T_{i,B} \cdot (1-b_i)$, where $T_{j,A}$ and $T_{i,B}$ take the form
        \[
        T_{j,A}(x) = \alpha_j\langle x, u_j \rangle v_j \quad \mbox{and} \quad T_{i,B}(x) = \beta_i\langle x, w_i \rangle z_i
        \]
        with $x\in \ell^2(G)$, $a_j\in A$, $\alpha_j\in \C$ and $u_j,v_j\in \ell^2(G)$ for all $j=1,\ldots, n$, and $b_i\in B$, $\beta_i \in \C$ and $w_i,z_i \in \ell^2(G)$ for all $i=1,\ldots, l$. We fix the same left transversals for $C$ in $A$ and $B$ as before, both containing $1$, in order to write the elements of $G$ in their normal form. Without loss of generality, we can assume that the elements $a_j$ are different from $1$ and belong to the left transversal for $C$ in $A$. Indeed, if $a_j = 1$ then $1 - a_j = 0$, and if $a_j = a_j'c$ for $a_j'$ in the left transversal, then
        \[
        T_{j,A} \cdot (1 - a_j) = T_{j,A} \cdot a_j'(1 -c) + T_{j,A} \cdot (1 - a_j')
        \]
        and $(T_{j,A} \cdot a_j') \cdot (1 -c)$ already lies in $F(\ell^2(G))\cdot I_{\scalebox{0.8}{$\C[C]$}}$. Note that $\sum_{a \in A} \langle T(a), a \rangle$ is finite by the absolute convergence of the trace (see \cref{subsec: finite_rank_op}). Moreover, the sum is zero since for each $j = 1, \ldots, n$, it holds that
        \begin{align*}
            \sum_{a \in A} \langle T_{j,A}(a), a \rangle &= \sum_{a \in A} \langle T_{j,A}(a_j^{-1} a), a_j^{-1} a \rangle = \sum_{a \in A} \langle R_{a_j} T_{j,A} R_{a_j^{-1}} (a), a \rangle \\
            & = \sum_{a \in A} \langle (T_{j,A}\cdot a_j)(a), a \rangle.
        \end{align*}
        Similarly, $\sum_{b \in B} \langle T(b), b \rangle$ is zero. We shall use the operator $D_{\widetilde{a}}$ to produce a zero sum running over $C$. From this, along with the fact that $C$ is finite, it will follow that $T \in F(\ell^2(G))\cdot I_{\scalebox{0.8}{$\C[C]$}}$. 
        
        Let $L_{g_1}$ and $L_{g_2}$ be left multiplication operators for arbitrary elements $g_1$ and $g_2$ of $G$. Note that for any $g\in G$ and $S\in F(\ell^2(G))$ it holds
        \[
        \trace(S\cdot (1-g)L_{g_1}D_{\widetilde{a}}L_{g_2})=\trace(SL_{g_1}D_{\widetilde{a}}L_{g_2})-\trace(R_{g}SR_{g^{-1}}L_{g_1}D_{\widetilde{a}}L_{g_2}).
        \]
        Moreover, since left and right multiplication commute, taking the trace over the standard basis of $\ell^2(G)$
        \[
        \trace(R_{g} S R_{g^{-1}} L_{g_1} D_{\widetilde{a}} L_{g_2}) = \trace(SL_{g_1}R_{g^{-1}}D_{\widetilde{a}}L_{g_2}R_{g})=\trace(SL_{g_1}R_{g^{-1}}D_{\widetilde{a}}R_{g}L_{g_2})
        \]
        Hence,
        \[
        \trace(S\cdot (1-g)L_{g_1}D_{\widetilde{a}}L_{g_2})=\trace(SL_{g_1}(D_{\widetilde{a}}-R_{g^{-1}}D_{\widetilde{a}}R_{g})L_{g_2}).
        \]
        Recall that by assumption $T= \sum_{i=1}^l T_{i,B} \cdot (1-b_i)$. Thus, in view of the first part of \cref{lem: proj_operator_amal_product}, we have that
        \[
        \trace(TL_{g_1}D_{\widetilde{a}}L_{g_2}) = \sum_{i=1}^l \trace(T_{i,B} \cdot (1-b_i)L_{g_1}D_{\widetilde{a}}L_{g_2}) = 0
        \]
        Alternatively, $T= \sum_{j=1}^n T_{j,A} \cdot (1-a_j)$. So the second part of \cref{lem: proj_operator_amal_product} yields
        \begin{align*}
            \trace(TL_{g_1}D_{\widetilde{a}}L_{g_2}) & = \sum_{j=1}^n \trace(T_{j,A} L_{g_1} (\pr_{C} L_{\widetilde{a}^{-1}} - R_{a_j^{-1}} \pr_{C} L_{\widetilde{a}^{-1}} R_{a_j}) L_{g_2}) \\ 
            &= \trace\left(\sum_{j=1}^n T_{j,A}\cdot (1-a_j) L_{g_1} \pr_{C} L_{\widetilde{a}^{-1}} L_{g_2} \right)\\
            &= \trace\left(L_{g_2} T L_{g_1} \pr_{C} L_{\widetilde{a}^{-1}} \right)
        \end{align*}
        where once again the trace was taken over the standard basis of $\ell^2(G)$. To sum up, we get that $\trace\left(L_{g_2} T L_{g_1} \pr_{C} L_{\widetilde{a}^{-1}} \right) = 0$. Now, it holds that
        \begin{align*}
            \sum_{g\in G} \langle L_{g_2} T L_{g_1} \pr_{C} L_{\widetilde{a}^{-1}}(g),g\rangle &= \sum_{g\in G} \langle T L_{g_1} \pr_{C} L_{\widetilde{a}^{-1}}(g),L_{g_2}^*(g)\rangle\\
            &=\sum_{c\in C} \langle T L_{g_1}(c),L_{g_2^{-1}}(\widetilde{a}c)\rangle.
        \end{align*}
        Furthermore, $L_{g_1}(c) = R_{c} L_{g_1}(1)$ and $L_{g_2^{-1}}(c) = R_{c} L_{g_2^{-1}}(1)$. Not only that, $C$ is finite, and hence it makes sense to consider the operator $\sum_{c\in C} T\cdot c$. Thus putting all together
        \begin{align*}
            0 = \sum_{c\in C} \langle T L_{g_1}(c),L_{g_2^{-1}}(\widetilde{a}c)\rangle & = \sum_{c\in C} \langle T R_{c} L_{g_1}(1), R_{c} L_{g_2^{-1}}(\widetilde{a})\rangle \\
            &= \sum_{c\in C} \langle R_{c^{-1}} T R_{c} L_{g_1}(1), L_{g_2^{-1}}(\widetilde{a})\rangle \\
            & = \langle (\sum_{c\in C} T\cdot c) L_{g_1}(1), L_{g_2^{-1}}(\widetilde{a}) \rangle.
        \end{align*}
        Since $g_1$ and $g_2$ are arbitrary elements in $G$, we conclude that $\sum_{c\in C} T\cdot c$ equals to zero. In other words, $T = -\sum_{c\in C\setminus\{1\}} T\cdot c$. Therefore
        \[
        T = \frac{1}{|C|} \sum_{c\in C\setminus\{1\}} T\cdot (1-c) \in F(\ell^2(G))\cdot I_{\scalebox{0.8}{$\C[C]$}}
        \]
        as we wanted.
    \end{proof}

    Let us set up our framework for the rest of the section. The following is also valid for a subfield $K \subseteq \C$ closed under complex conjugation, though we restrict to $\C$ for the sake of clarity. Consider a group $G$ that is the amalgamated free product of two groups $A$ and $B$ over $C$. We denote by $\mathcal{R}_{\scalebox{0.8}{$\C[A]$}}$, $\mathcal{R}_{\scalebox{0.8}{$\C[B]$}}$ and $\mathcal{R}_{\scalebox{0.8}{$\C[C]$}}$ the $\ast$-regular closure of $\C[A]$, $\C[B]$ and $\C[C]$ in $\UO(A)$, $\UO(B)$ and $\UO(C)$, respectively. Recall that $\mathcal{R}_{\scalebox{0.8}{$\C[C]$}} \subseteq \mathcal{R}_{\scalebox{0.8}{$\C[A]$}}$ and $\mathcal{R}_{\scalebox{0.8}{$\C[C]$}} \subseteq \mathcal{R}_{\scalebox{0.8}{$\C[B]$}}$ (see, for instance, \cite[Lemma 4.1.10]{DLopezAlvarezThesis}). We define the coproduct ring $R_{G}:=\mathcal{R}_{\scalebox{0.8}{$\C[A]$}} \ast_{\scalebox{0.8}{$\mathcal{R}_{\scalebox{0.8}{$\C[C]$}}$}} \mathcal{R}_{\scalebox{0.8}{$\C[B]$}}$. Note that $\C[G]$ sits naturally inside $R_G$. Moreover, according to the universal property of coproducts, there is a ring homomorphism $\phi: R_G\rightarrow \UO(G)$ induced by the natural inclusions $\mathcal{R}_{\scalebox{0.8}{$\C[A]$}}\subseteq \UO(A)\subseteq \UO(G)$ and $\mathcal{R}_{\scalebox{0.8}{$\C[B]$}}\subseteq \UO(B)\subseteq \UO(G)$ (cf. \cite[Eq. (8.12)]{Luck02}). So we have the following commutative diagram of ring homomorphisms
    \[
    \begin{tikzcd}
        R_G \arrow[rr, "\phi"]&  & \UO(G)\\
        & \C[G] \arrow[ur, hook] \arrow[lu, hook']& \nospacepunct{.} 
    \end{tikzcd}
    \]
    Note that the image of $\phi$ is contained in $\mathcal{R}_{\scalebox{0.8}{$\C[G]$}}$. Recall also that $\phi$ induces an $\SMRF$ on $R_G$. For the next result, we use the interpretation of finite rank operators given in the proof of \cite[Theorem 1.2]{Jaikin_ExplicitUFree}. 

    \begin{prop}\label{prop: vanishing_Tor1_amalg}
        Let $A$ and $B$ be two groups, $C$ a finite group and let $G=A \ast_C B$ be their amalgamated free product. With the above notation, it holds that
        \[
        \Tor_1^{R_{G}}(\UO(G),\UO(G)) = 0.
        \]
        In particular, $\Tor_1^{R_{G}}(\mathcal{R}_{\scalebox{0.8}{$\C[G]$}}, \mathcal{R}_{\scalebox{0.8}{$\C[G]$}}) = 0$.
    \end{prop}
    \begin{proof}
        Since $R_{G}$ is a coproduct of rings
        \begin{align*}
            &\Tor_1^{\scalebox{0.8}{$\mathcal{R}_{\scalebox{0.8}{$\C[A]$}}$}}(\UO(G),\UO(G)) \oplus \Tor_1^{\scalebox{0.8}{$\mathcal{R}_{\scalebox{0.8}{$\C[B]$}}$}}(\UO(G),\UO(G)) \rightarrow \Tor_1^{R_{G}}(\UO(G),\UO(G)) \\
            & \rightarrow \UO(G)\otimes_{\scalebox{0.8}{$\mathcal{R}_{\scalebox{0.8}{$\C[C]$}}$}} \UO(G) \rightarrow \UO(G)\otimes_{\scalebox{0.8}{$\mathcal{R}_{\scalebox{0.8}{$\C[A]$}}$}} \UO(G) \oplus \UO(G)\otimes_{\scalebox{0.8}{$\mathcal{R}_{\scalebox{0.8}{$\C[B]$}}$}} \UO(G).
        \end{align*}
        is an exact sequence according to \cite[Theorem 6]{Dicks_MayerVietoris}. By definition, both $\mathcal{R}_{\scalebox{0.8}{$\C[A]$}}$ and $\mathcal{R}_{\scalebox{0.8}{$\C[B]$}}$ are von Neumann regular rings, so $\UO(G)$ is a flat module and their corresponding $\Tor_1$ in the first row vanish. Thus to show that $\Tor_1^{R_{G}}(\UO(G),\UO(G))=0$, it suffices to prove that the diagonal map
        \[
        \UO(G)\otimes_{\scalebox{0.8}{$\mathcal{R}_{\scalebox{0.8}{$\C[C]$}}$}} \UO(G)\rightarrow \UO(G)\otimes_{\scalebox{0.8}{$\mathcal{R}_{\scalebox{0.8}{$\C[A]$}}$}} \UO(G) \oplus \UO(G)\otimes_{\scalebox{0.8}{$\mathcal{R}_{\scalebox{0.8}{$\C[B]$}}$}} \UO(G)
        \]
        is injective. To this end, we shall apply several reductions. First, recall that the inclusions $\C[A]\subseteq \mathcal{R}_{\scalebox{0.8}{$\C[A]$}}$, $\C[B]\subseteq \mathcal{R}_{\scalebox{0.8}{$\C[B]$}}$ and $\C[C]\subseteq \mathcal{R}_{\scalebox{0.8}{$\C[C]$}}$ are epic. So then, according to \cref{prop: epic_tensor_prod}, it is enough to show that the diagonal map
        \[
        \UO(G)\otimes_{\scalebox{0.8}{$\C[C]$}} \UO(G)\rightarrow \UO(G)\otimes_{\scalebox{0.8}{$\C[A]$}} \UO(G) \oplus \UO(G)\otimes_{\scalebox{0.8}{$\C[B]$}} \UO(G)
        \]
        is injective. Recall that $\UO(G)$ is the left (resp. right) Ore localization of $\VN(G)$. Note that, since $\VN(G) \subseteq \ell^2(G) \subseteq \UO(G)$, every $x \in \ell^2(G)$ can be written as $s^{-1}t$ for $s,t \in \VN(G)$ and $s$ not a zero-divisor. Hence, over $\UO(G)\otimes_{\scalebox{0.8}{$\VN(G)$}}\ell^2(G)$ it holds that $1 \otimes x = s^{-1}s \otimes x = s^{-1} \otimes t = x \otimes 1$, and so, the multiplication map $\UO(G)\otimes_{\scalebox{0.8}{$\VN(G)$}}\ell^2(G) \rightarrow \UO(G)$ is an isomorphism of $\UO(G)$-$\C[G]$-bimodules. Similarly, $\UO(G)$ is isomorphic to $\ell^2(G)\otimes_{\scalebox{0.8}{$\VN(G)$}}\UO(G)$ as $\C[G]$-$\UO(G)$-bimodules. Not only that, $\UO(G)$ is a flat right (resp. left) $\VN(G)$-module by \cite[Corollary 10.13]{GoodWar04}, and thus taking tensor product with $\UO(G)$ over $\VN(G)$ preserves injective maps. Then, putting all together, it suffices to prove that the diagonal map
        \[
        \ell^2(G)\otimes_{\scalebox{0.8}{$\C[C]$}} \ell^2(G)\rightarrow \ell^2(G)\otimes_{\scalebox{0.8}{$\C[A]$}} \ell^2(G) \oplus \ell^2(G)\otimes_{\scalebox{0.8}{$\C[B]$}} \ell^2(G)
        \]
        is injective. Let us give an operator interpretation to this map. On the one side, the space $F(\ell^2(G))$ is isomorphic as $\C$-vector space to $\ell^2(G)\otimes_\C \ell^2(G)$ with an explicit isomorphism $\alpha:\ell^2(G)\otimes_\C \ell^2(G)\rightarrow F(\ell^2(G))$ given by
        \[
        \alpha(v\otimes u)(x):=\langle x, u^* \rangle v = T_{u^*,v}(x)
        \]
        for $u,v,x \in \ell^2 (G)$. On the other hand, $\ell^2(G)\otimes_{\scalebox{0.8}{$\C[A]$}} \ell^2(G)$ is isomorphic to the quotient of $\ell^2(G)\otimes_{\C} \ell^2(G)$ by the $\C$-vector space generated by the elements of the form 
        \[
        va\otimes u - v\otimes au
        \]
        for $u,v\in \ell^2(G)$ and $a\in A$. Note that $\alpha$ sends $va\otimes u - v\otimes au$ to the finite rank operator $T_{u^{\ast},va} \cdot (1-a^{-1})$. So, in terms of finite rank operators, this quotient of $\C$-vector spaces corresponds to the quotient $F(\ell^2(G))/F(\ell^2(G))\cdot I_{\scalebox{0.8}{$\C[A]$}}$. Similarly, $\ell^2(G)\otimes_{\scalebox{0.8}{$\C[B]$}} \ell^2(G)$ and $\ell^2(G)\otimes_{\scalebox{0.8}{$\C[C]$}} \ell^2(G)$ are isomorphic as $\C$-vector spaces to $F(\ell^2(G))/F(\ell^2(G))\cdot I_{\scalebox{0.8}{$\C[B]$}}$ and $F(\ell^2(G))/F(\ell^2(G))\cdot I_{\scalebox{0.8}{$\C[C]$}}$, respectively. To put it another way, the map
        \[
        \ell^2(G)\otimes_{\scalebox{0.8}{$\C[C]$}} \ell^2(G)\rightarrow \ell^2(G)\otimes_{\scalebox{0.8}{$\C[A]$}} \ell^2(G) \oplus \ell^2(G)\otimes_{\scalebox{0.8}{$\C[B]$}} \ell^2(G)
        \]
        is injective if and only if the diagonal map
        \[
        F(\ell^2(G))/F(\ell^2(G))\cdot I_{\scalebox{0.8}{$\C[C]$}} \rightarrow F(\ell^2(G))/F(\ell^2(G))\cdot I_{\scalebox{0.8}{$\C[A]$}} \oplus F(\ell^2(G))/F(\ell^2(G))\cdot I_{\scalebox{0.8}{$\C[B]$}}
        \]
        is injective. But, according to \cref{lem: amalg_prod_intersection_augment_ideal} the latter map is injective, and hence, putting all together, we get that $\Tor_1^{R_{G}}(\UO(G),\UO(G))=0$.

        For the last statement, we invoke the exact sequence of \cite[Theorem 6]{Dicks_MayerVietoris} for $\mathcal{R}_{\scalebox{0.8}{$\C[G]$}}$ instead of $\UO(G)$ and use the vanishing of $\Tor_1^{R_{G}}(\UO(G),\UO(G))$ along with the inclusion
        \[
        \mathcal{R}_{\scalebox{0.8}{$\C[G]$}}\otimes_{\scalebox{0.8}{$\mathcal{R}_{\scalebox{0.8}{$\C[C]$}}$}} \mathcal{R}_{\scalebox{0.8}{$\C[G]$}} \subseteq \UO(G)\otimes_{\scalebox{0.8}{$\mathcal{R}_{\scalebox{0.8}{$\C[C]$}}$}} \UO(G).
        \]
    \end{proof}
    
    Now we are ready to prove the first main result of the section. 

    \begin{thm}\label{thm: SAC_amal_prod}
        Let $A$ and $B$ be two groups satisfying the strong Atiyah conjecture over $\C$ and $C$ a finite group. Then the amalgamated free product $G = A \ast_C B$ satisfies the strong Atiyah conjecture over $\C$.
    \end{thm}
    \begin{proof}
        Observe that if $M$ is a finitely presented left $\C[G]$-module, then
        \[
        \dim_{\scalebox{0.8}{$\UO(G)$}}(\UO(G)\otimes_{\scalebox{0.8}{$\C[G]$}} M) = \dim_{\scalebox{0.8}{$\UO(G)$}}(\UO(G) \otimes_{R_{G}} R_{G} \otimes_{\scalebox{0.8}{$\C[G]$}} M)
        \]
        Therefore, to prove the strong Atiyah conjecture for $\C[G]$ it suffices to show that for any finitely presented left $R_{G}$-module $N$
        \[
        \dim_{\scalebox{0.8}{$\UO(G)$}}(\UO(G)\otimes_{R_{G}} N) \in \frac{1}{\lcm(G)}\Z
        \]
        since $R_{G} \otimes_{\scalebox{0.8}{$\C[G]$}} M$ is finitely presented as left $R_{G}$-module. So let $N$ be a finitely presented left $R_{G}$-module. Clearly, $\UO(G) \otimes_{R_{G}} N$ is also finitely presented as left $\UO(G)$-module, and hence, it is a direct sum of cyclic projective $\UO(G)$-modules by \cref{prop: vNr_proj_mod}. In particular, $\UO(G) \otimes_{R_{G}} N$ embeds in $\UO(G)^k$ for some positive integer $k$. Let $\overline{N}$ denote the image of $N$ in $\UO(G) \otimes_{R_{G}} N$ sitting inside $\UO(G)^k$. Note that $\overline{N}$ is an $R_{G}$-submodule of $\UO(G)^k$.

        On the other hand, $\overline{N}$ is finitely generated, so it admits a presentation of left $R_{G}$-modules of the form $0\rightarrow I \rightarrow R_{G}^d \rightarrow \overline{N} \rightarrow 0$ for some positive integer $d$. Observe that $R_{G}$ has weak dimension at most $1$ since by \cref{thm: induced_submod} every ideal is an induced module over semisimple rings according to \cref{prop: Atiyah_implies_semisimple}. In particular, $\Tor_1^{R_{G}}(\UO(G), \overline{N})$ also sits inside $\Tor_1^{R_{G}}(\UO(G), \UO(G)^k)$. Therefore, using that $\Tor_1^{R_{G}}(\UO(G), \UO(G)) = 0$ by \cref{prop: vanishing_Tor1_amalg}, we conclude that $\Tor_1^{R_{G}}(\UO(G), \overline{N}) = 0$. Consequently, if we take tensor product with $\UO(G)\otimes_{R_{G}}$ in the above presentation, we obtain the next exact sequence of left $\UO(G)$-modules
        \[
        0 \rightarrow \UO(G)\otimes_{R_{G}} I \rightarrow \UO(G)^d \rightarrow \UO(G)\otimes_{R_{G}} \overline{N} \rightarrow 0.
        \]
        Since $\UO(G)$ is von Neumann regular, the dimension is exact, and in particular, $\dim_{\scalebox{0.8}{$\UO(G)$}}(\UO(G)\otimes_{R_G} I)$ is finite. Hence, taking dimensions yields
        \[
        \dim_{\scalebox{0.8}{$\UO(G)$}}(\UO(G)\otimes_{R_{G}} N) = \dim_{\scalebox{0.8}{$\UO(G)$}}(\UO(G)\otimes_{R_{G}} \overline{N}) = d - \dim_{\scalebox{0.8}{$\UO(G)$}}(\UO(G)\otimes_{R_{G}} I).
        \]
        In addition, $I$ is a left $R_{G}$-submodule of the induced module $R_{G}^d$. So, according to \cref{thm: induced_submod}, $I$ must have an induced module structure, say
        \[
        I = R_{G} \otimes_{\scalebox{0.8}{$\mathcal{R}_{\scalebox{0.8}{$\C[A]$}}$}} I_1 \oplus R_{G} \otimes_{\scalebox{0.8}{$\mathcal{R}_{\scalebox{0.8}{$\C[B]$}}$}} I_2
        \]
        where $I_1$ and $I_2$ are left $\mathcal{R}_{\scalebox{0.8}{$\C[A]$}}$ and $\mathcal{R}_{\scalebox{0.8}{$\C[B]$}}$-modules, respectively. Since tensor products commute with direct sums and $\dim_{\scalebox{0.8}{$\UO(G)$}}$ is additive, we get
        \[
        \dim_{\scalebox{0.8}{$\UO(G)$}}(\UO(G)\otimes_{R_{G}} I) = \dim_{\scalebox{0.8}{$\UO(G)$}}(\UO(G)\otimes_{\scalebox{0.8}{$\mathcal{R}_{\scalebox{0.8}{$\C[A]$}}$}} I_1) + \dim_{\scalebox{0.8}{$\UO(G)$}}(\UO(G)\otimes_{\scalebox{0.8}{$\mathcal{R}_{\scalebox{0.8}{$\C[B]$}}$}} I_2).
        \]
        But $\mathcal{R}_{\scalebox{0.8}{$\C[A]$}} \subseteq \UO(A)$ and $\mathcal{R}_{\scalebox{0.8}{$\C[B]$}} \subseteq \UO(B)$, and then, by \cref{prop: preserve_dim_gps}, this expression can be written in terms of the groups $A$ and $B$ as follows
        \[
        \dim_{\scalebox{0.8}{$\UO(G)$}}(\UO(G)\otimes_{R_{G}} I) = \dim_{\scalebox{0.8}{$\UO(A)$}}(\UO(A)\otimes_{\scalebox{0.8}{$\mathcal{R}_{\scalebox{0.8}{$\C[A]$}}$}} I_1) + \dim_{\scalebox{0.8}{$\UO(B)$}}(\UO(B)\otimes_{\scalebox{0.8}{$\mathcal{R}_{\scalebox{0.8}{$\C[B]$}}$}} I_2).
        \]
        Finally, since the strong Atiyah conjecture holds for $A$ and $B$ over $\C$ and all the dimensions are finite, \cref{prop: Atiyah_conj_reg_closure} implies that
        \[
        \dim_{\scalebox{0.8}{$\UO(A)$}}(\UO(A)\otimes_{\scalebox{0.8}{$\mathcal{R}_{\scalebox{0.8}{$\C[A]$}}$}} I_1)\in \frac{1}{\lcm(A)}\Z \quad \mbox{and} \quad \dim_{\scalebox{0.8}{$\UO(B)$}}(\UO(B)\otimes_{\scalebox{0.8}{$\mathcal{R}_{\scalebox{0.8}{$\C[B]$}}$}} I_2)\in \frac{1}{\lcm(B)}\Z.
        \]
        That is to say,
        \[
        \dim_{\scalebox{0.8}{$\UO(G)$}}(\UO(G)\otimes_{R_{G}} I) \in \frac{1}{\lcm(G)}\Z,
        \]
        and so does $\dim_{\scalebox{0.8}{$\UO(G)$}}(\UO(G)\otimes_{R_{G}} N)$ which concludes the proof.
    \end{proof}

    \begin{rem}\label{rem: sAc_arbitrary_amal_prod}
        The strong Atiyah conjecture needs only to be checked over finitely presented modules, or by duality over matrices that have finitely many entries. So it follows by induction from \cref{thm: SAC_amal_prod}, that the strong Atiyah conjecture is also closed under arbitrarily many consecutive amalgamated free products over finite groups.
    \end{rem}

    The strong Atiyah conjecture endows $\mathcal{R}_{\scalebox{0.8}{$\C[G]$}}$ with the semisimple ring structure. We take advantage of this feature and proceed to show that $\mathcal{R}_{\scalebox{0.8}{$\C[G]$}}$ is a universal localization of $R_G$. Later, we shall use this fact to prove that the algebraic Atiyah conjecture is also closed under taking amalgamations along finite groups. Let us record an intermediate result.

    \begin{prop} \label{prop: R_G_hereditary_faithful}
        Let $A$ and $B$ be two groups satisfying the strong Atiyah conjecture over $\C$ and $C$ a finite group, and let $G = A \ast_C B$ be their amalgamated free product. Then $R_G$ is a hereditary ring and the $\SMRF$ induced from $\UO(G)$ via $\phi$ is faithful over finitely generated projective left $R_G$-modules.
    \end{prop}
    \begin{proof}
        Since $A$, $B$ and $C$ satisfy the strong Atiyah conjecture over $\C$, then the corresponding $\ast$-regular closures are semisimple rings by \cref{prop: Atiyah_implies_semisimple}. Thus, it follows from \cref{thm: induced_submod} that $R_G$ is a hereditary ring. 

        For the second part, let $P$ be a finitely generated projective left $R_G$-module. According to \cref{prop: proj_mod_GR} $P$ is of the form
        \[
        R_{G} \otimes_{\scalebox{0.8}{$\mathcal{R}_{\scalebox{0.8}{$\C[A]$}}$}} P_1 \oplus R_{G} \otimes_{\scalebox{0.8}{$\mathcal{R}_{\scalebox{0.8}{$\C[B]$}}$}} P_2
        \]
        where $P_1$ and $P_2$ are projective left $\mathcal{R}_{\scalebox{0.8}{$\C[A]$}}$ and $\mathcal{R}_{\scalebox{0.8}{$\C[B]$}}$-modules, respectively. Furthermore, both of them are finitely generated by \cite[Lemma 4]{LinnellLuckSchick_OreCond}. Thus if $\dim_{\scalebox{0.8}{$\UO(G)$}}(\UO(G)\otimes_{R_{G}} P) = 0$, it holds that
        \[
        \dim_{\scalebox{0.8}{$\UO(G)$}}(\UO(G)\otimes_{\scalebox{0.8}{$\mathcal{R}_{\scalebox{0.8}{$\C[A]$}}$}} P_1) = 0 = \dim_{\scalebox{0.8}{$\UO(G)$}}(\UO(G)\otimes_{\scalebox{0.8}{$\mathcal{R}_{\scalebox{0.8}{$\C[B]$}}$}} P_2).
        \]
        However, from \cref{prop: dimG_proj_faithful} along with \cref{prop: preserve_dim_gps}, we deduce that $P_1$ and $P_2$ are the zero module, and hence, so is $P$.
    \end{proof}

    \begin{rem} \label{rem: K_0_R_G}
        Note that the proof actually shows that the map
        \[
        K_0(\mathcal{R}_{\scalebox{0.8}{$\C[A]$}}) \oplus K_0(\mathcal{R}_{\scalebox{0.8}{$\C[B]$}}) \rightarrow K_0(R_G)
        \]
        is surjective.
    \end{rem}

    \begin{thm}\label{thm: univ_local_amalg_prod}
        Let $A$ and $B$ be two groups satisfying the strong Atiyah conjecture over $\C$ and $C$ a finite group, and let $G = A \ast_C B$ be their amalgamated free product. Then $\mathcal{R}_{\scalebox{0.8}{$\C[G]$}}$ is the universal localization of all full maps over $R_G$ with respect to the $L^2$-dimension.
    \end{thm}
    \begin{proof}
        We show that we are under the hypothesis of \cref{thm: recogn_universal_localization}. From \cref{prop: R_G_hereditary_faithful} we already know that $R_G$ is a hereditary ring and that the induced $\SMRF$ via $\phi$, that is the pull-back of the $L^2$-dimension, is faithful over finitely generated projective left $R_G$-modules. The map $\phi : R_G \rightarrow \mathcal{R}_{\scalebox{0.8}{$\C[G]$}}$ is clearly epic since $\C[G] \subseteq \mathcal{R}_{\scalebox{0.8}{$\C[G]$}}$ is epic. In addition, according to \cref{thm: SAC_amal_prod} $G$ satisfies the strong Atiyah conjecture over $\C$. In particular, $\mathcal{R}_{\scalebox{0.8}{$\C[G]$}}$ is a semisimple ring by \cref{prop: Atiyah_implies_semisimple}. Finally, \cref{prop: vanishing_Tor1_amalg} states that 
        \[
        \Tor_1^{\scalebox{0.8}{$R_G$}}(\mathcal{R}_{\scalebox{0.8}{$\C[G]$}},\mathcal{R}_{\scalebox{0.8}{$\C[G]$}}) = 0.
        \]
        Therefore, applying \cref{thm: recogn_universal_localization} to $R_G$ and $\mathcal{R}_{\scalebox{0.8}{$\C[G]$}}$ ends the proof.
    \end{proof}

    \begin{cor}\label{cor: K0_amalgamated}
        Let $A$ and $B$ be two groups satisfying the strong Atiyah conjecture over $\C$ and $C$ a finite group, and let $G = A \ast_C B$ be their amalgamated free product. Then the map
        \[
        K_0(\mathcal{R}_{\scalebox{0.8}{$\C[A]$}}) \oplus K_0(\mathcal{R}_{\scalebox{0.8}{$\C[B]$}}) \rightarrow K_0(\mathcal{R}_{\scalebox{0.8}{$\C[G]$}})
        \]
        is surjective.
    \end{cor}
    \begin{proof}
        According to \cref{prop: R_G_hereditary_faithful}, $R_G$ is a hereditary ring with an $\SMRF$ that is faithful over finitely generated projective left $R_G$-modules. Furthermore, $\mathcal{R}_{\scalebox{0.8}{$\C[G]$}}$ is a universal localization of $R_G$ at full maps by \cref{thm: univ_local_amalg_prod}. Thus, by \cref{prop: project_mod_localization} we can conclude that the map $K_0(R_G) \rightarrow K_0(\mathcal{R}_{\scalebox{0.8}{$\C[G]$}})$ is surjective. But, recall that the map
        \[
        K_0(\mathcal{R}_{\scalebox{0.8}{$\C[A]$}}) \oplus K_0(\mathcal{R}_{\scalebox{0.8}{$\C[B]$}}) \rightarrow K_0(R_G)
        \]
        is surjective (see \cref{rem: K_0_R_G}). So, composing both surjections we get the desired result.
    \end{proof}

    \begin{thm}\label{thm: AAC_amal_prod}
        Let $A$ and $B$ be two groups satisfying the algebraic Atiyah conjecture over $\C$ and $C$ a finite group. Then the amalgamated free product $G = A \ast_C B$ satisfies the algebraic Atiyah conjecture over $\C$.
    \end{thm}
    \begin{proof}
        According to \cref{cor: K0_amalgamated}, the map
        \[
        K_0(\mathcal{R}_{\scalebox{0.8}{$\C[A]$}}) \oplus K_0(\mathcal{R}_{\scalebox{0.8}{$\C[B]$}}) \rightarrow K_0(\mathcal{R}_{\scalebox{0.8}{$\C[G]$}})
        \]
        is surjective as long as $A$ and $B$ satisfy the strong Atiyah conjecture over $\C$. But this is the case since both of them satisfy the algebraic Atiyah conjecture over $\C$. By definition, the algebraic Atiyah maps corresponding to $A$ and $B$ are surjective. So the algebraic Atiyah map of $G$ is also surjective.
    \end{proof}

\section{HNN-extensions over finite groups} \label{sec: HNN_ext}

    In this section we focus on HNN-extensions over finite groups and prove similar results to the ones in the amalgamated free product case. Although we could have adapted the proof of \cref{prop: vanishing_Tor1_amalg} to make it work in the HNN-extensions setting, we have decided to carry out a different approach. Concretely, we present the group given by the HNN-extension as a cyclic extension of a normal subgroup that is defined as the fundamental group of a tree of amalgamated groups over finite groups, so that we can use the machinery of the previous section on the normal subgroup and prove the analog of \cref{prop: vanishing_Tor1_amalg}. From this, arguing as in \cref{sec: amalg_free_prod}, it follows at once that both the strong and algebraic Atiyah conjectures are closed under HNN-extensions over finite groups as well as the fact that the $\ast$-regular closure of the group is the universal localization over all full maps of the associated HNN of rings construction. 
    
    Let us set up our new framework for this section. As before, the following is also valid for a subfield $K \subseteq \C$ closed under complex conjugation, but we have decided to work with $\C$ to keep the exposition as clear as possible. Consider a group $G$ that is the HNN-extension of the group $A$ over $C$ with associated map $\psi$. We denote by $\mathcal{R}_{\scalebox{0.8}{$\C[A]$}}$ and $\mathcal{R}_{\scalebox{0.8}{$\C[C]$}}$ the $\ast$-regular closure of $\C[A]$ and $\C[C]$ in $\UO(A)$ and $\UO(C)$, respectively. Recall that $\mathcal{R}_{\scalebox{0.8}{$\C[C]$}} \subseteq \mathcal{R}_{\scalebox{0.8}{$\C[A]$}}$ (cf. \cite[Lemma 4.1.10]{DLopezAlvarezThesis}) and observe that $\psi$ extends to a ring isomorphism $\widetilde{\psi}$ between $\mathcal{R}_{\scalebox{0.8}{$\C[C]$}}$ and $\mathcal{R}_{\scalebox{0.8}{$\C[\psi(C)]$}}$ (for instance by the explicit construction of the $\ast$-regular closure). Then, we define the HNN ring $R_{G}:=\mathcal{R}_{\scalebox{0.8}{$\C[A]_{\scalebox{0.8}{$\mathcal{R}_{\scalebox{0.8}{$\C[C]$}}$}}$}} \langle t, t^{-1} ; \widetilde{\psi} \rangle$ (see \cref{subsec: graph_rings}). Note that $\C[G]$ sits naturally inside $R_G$. Moreover, we have a ring homomorphism $\phi: R_G\rightarrow \UO(G)$ induced by the natural inclusions $\mathcal{R}_{\scalebox{0.8}{$\C[A]$}}\subseteq \UO(A)\subseteq \UO(G)$ (cf. \cite[Eq. (8.12)]{Luck02}). So we can consider the next commutative diagram of ring homomorphisms
    \[
    \begin{tikzcd}
        R_G \arrow[rr, "\phi"]&  & \UO(G)\\
        & \C[G] \arrow[ur, hook] \arrow[lu, hook']& \nospacepunct{.} 
    \end{tikzcd}
    \]
    As before, the image of $\phi$ is contained in $\mathcal{R}_{\scalebox{0.8}{$\C[G]$}}$. Note that $G$ is an extension of the normal subgroup $N$ defined as the fundamental group of the tree of amalgamated groups
    \[
    \ldots tAt^{-1} \ast_C A \ast_{t^{-1}Ct} t^{-1}At \ldots
    \]
    and the infinite cyclic group generated by the stable letter $t$. Similarly, $R_G$ can be written as the skew Laurent polynomial ring $S_N[t,t^{-1}; \sigma]$, where $S_N$ is the countable tree coproduct
    \[
    \ldots t \mathcal{R}_{\scalebox{0.8}{$\C[A]$}} t^{-1} \ast_{\scalebox{0.8}{$\mathcal{R}_{\scalebox{0.8}{$\C[C]$}}$}} \mathcal{R}_{\scalebox{0.8}{$\C[A]$}} \ast_{\scalebox{0.8}{$t^{-1} \mathcal{R}_{\scalebox{0.8}{$\C[C]$}} t$}} t^{-1} \mathcal{R}_{\scalebox{0.8}{$\C[A]$}} t \ldots
    \]
    and $\sigma$ is the right-shift automorphism corresponding to conjugation by $t$. Note that $t \mathcal{R}_{\scalebox{0.8}{$\C[A]$}} t^{-1} \subseteq \mathcal{R}_{\scalebox{0.8}{$\C[G]$}}$ is a $\ast$-regular ring containing $\C[tAt^{-1}]$. From the definition, it follows that $t \mathcal{R}_{\scalebox{0.8}{$\C[A]$}} t^{-1} = \mathcal{R}_{\scalebox{0.8}{$\C[tAt^{-1}]$}}$. The same applies to different powers of $t$ and the group $C$. In other words, $S_N$ is the tree coproduct ring
    \[
    \ldots \mathcal{R}_{\scalebox{0.8}{$\C[tAt^{-1}]$}} \ast_{\scalebox{0.8}{$\mathcal{R}_{\scalebox{0.8}{$\C[C]$}}$}} \mathcal{R}_{\scalebox{0.8}{$\C[A]$}} \ast_{\scalebox{0.8}{$\mathcal{R}_{\scalebox{0.8}{$\C[t^{-1}Ct]$}}$}} \mathcal{R}_{\scalebox{0.8}{$\C[t^{-1}At]$}} \ldots
    \]

    Before we move on to show that $\Tor_1^{\scalebox{0.8}{$R_G$}}(\mathcal{R}_{\scalebox{0.8}{$\C[G]$}}, \mathcal{R}_{\scalebox{0.8}{$\C[G]$}})$ vanishes, we record a variation of a well-known result.

    \begin{lem}\label{lem: localization_by_cyclic}
        Let $1 \rightarrow N \rightarrow G \rightarrow \Z \rightarrow 1$ be an exact sequence of groups. If $\mathcal{R}_{\scalebox{0.8}{$\C[N]$}}$ is a semisimple ring, then
        \[
        \mathcal{R}_{\scalebox{0.8}{$\C[G]$}} = \Ore(\mathcal{R}_{\scalebox{0.8}{$\C[N]$}} \ast G/N)
        \]
        and is a semisimple ring, where the Ore localization is carried out with respect to the set of non-zero-divisors.
    \end{lem}
    \begin{proof}
        First, $\mathcal{R}_{\scalebox{0.8}{$\C[N]$}} \ast G/N$ is a $\ast$-subring of $\UO(G)$, and hence semiprime by \cite[Lemma 2.4 (ii)]{LinnellDivRings93}. Second, $\mathcal{R}_{\scalebox{0.8}{$\C[N]$}}$ is noetherian being a semisimple ring, so $\mathcal{R}_{\scalebox{0.8}{$\C[N]$}} \ast G/N$ is noetherian as $G/N \cong \Z$ (see, for instance, \cite[Lemma 8.12]{ReichThesis}). Therefore, by Goldie's theorem \cite[Corollary 6.16]{GoodWar04} it makes sense to consider the Ore localization of $\mathcal{R}_{\scalebox{0.8}{$\C[N]$}} \ast G/N$ with respect to the set of non-zero-divisors and it is a semisimple ring. Now, to prove the equality, observe that $\Ore(\mathcal{R}_{\scalebox{0.8}{$\C[N]$}} \ast G/N)$ is a $\ast$-regular with the operation given by
        \[
        (u^{-1}r)^{\ast} = r^{\ast}(u^{\ast})^{-1}
        \]
        which is well-defined since the right Goldie quotient ring $\Ore(\mathcal{R}_{\scalebox{0.8}{$\C[N]$}} \ast G/N)$ is also a left Goldie quotient ring \cite[Proposition 6.20]{GoodWar04}. Hence, by the definition of the $\ast$-regular closure, it suffices to show that $\Ore(\mathcal{R}_{\scalebox{0.8}{$\C[N]$}} \ast G/N)$ embeds into $\mathcal{R}_{\scalebox{0.8}{$\C[G]$}}$. But this is equivalent to checking that the set of non-zero-divisors of $\mathcal{R}_{\scalebox{0.8}{$\C[N]$}} \ast G/N$ is invertible in $\UO(G)$. In fact this is the case according to \cite[Corollary 8.16 \& Proposition 8.17]{ReichThesis} (this is a variation of \cite[Lemma 4.4]{LinnellDivRings93}).
    \end{proof}

    \begin{prop} \label{prop: vanishing_Tor1_HNN}
        Let $A$ be a group satisfying the strong Atiyah conjecture over $\C$ and $C$ a finite group, and let $G = A \ast_C$ be their HNN-extension. Then $\mathcal{R}_{\scalebox{0.8}{$\C[N]$}}$ is a universal localization of $S_N$ and a semisimple ring. In particular, the modules $\Tor_1^{\scalebox{0.8}{$R_G$}}(\mathcal{R}_{\scalebox{0.8}{$\C[G]$}}, \mathcal{R}_{\scalebox{0.8}{$\C[G]$}})$ and $\Tor_1^{\scalebox{0.8}{$R_G$}}(\UO(G), \UO(G))$ vanish.
    \end{prop}
    \begin{proof}
        We start proving the consequence assuming the first part of the statement. Since $\mathcal{R}_{\scalebox{0.8}{$\C[G]$}}$ is a von Neumann regular ring, it suffices to check that the module $\Tor_1^{\scalebox{0.8}{$R_G$}}(\mathcal{R}_{\scalebox{0.8}{$\C[G]$}}, \mathcal{R}_{\scalebox{0.8}{$\C[G]$}}) = 0$. By \cref{lem: localization_by_cyclic} $\mathcal{R}_{\scalebox{0.8}{$\C[G]$}}$ is the Ore localization with respect to the set of non-zero-divisors of $\mathcal{R}_{\scalebox{0.8}{$\C[N]$}} \ast G/N$, so it is flat as left and right $\mathcal{R}_{\scalebox{0.8}{$\C[N]$}} \ast G/N$-module (cf. \cite[Corollary 10.13]{GoodWar04}). Thus $\Tor_1^{\scalebox{0.8}{$R_G$}}(\mathcal{R}_{\scalebox{0.8}{$\C[G]$}}, \mathcal{R}_{\scalebox{0.8}{$\C[G]$}})$ is the zero module if $\Tor_1^{\scalebox{0.8}{$R_G$}}(\mathcal{R}_{\scalebox{0.8}{$\C[N]$}} \ast G/N, \mathcal{R}_{\scalebox{0.8}{$\C[N]$}} \ast G/N)$ vanishes. On the other hand, the module $\Tor_1^{\scalebox{0.8}{$R_G$}}(\mathcal{R}_{\scalebox{0.8}{$\C[N]$}} \ast G/N, \mathcal{R}_{\scalebox{0.8}{$\C[N]$}} \ast G/N)$ is isomorphic to 
        \[
        \Tor_1^{\scalebox{0.8}{$S_N \otimes_{\scalebox{0.8}{$\C[N]$}} \C[G]$}}(\mathcal{R}_{\scalebox{0.8}{$\C[N]$}} \otimes_{\scalebox{0.8}{$\C[N]$}} \C[G], \mathcal{R}_{\scalebox{0.8}{$\C[N]$}} \otimes_{\scalebox{0.8}{$\C[N]$}} \C[G]),
        \]
        and this is just $\oplus_{|G/N|} \Tor_1^{\scalebox{0.8}{$S_N$}}(\mathcal{R}_{\scalebox{0.8}{$\C[N]$}}, \mathcal{R}_{\scalebox{0.8}{$\C[N]$}})$ by Shapiro's lemma. So, if $\mathcal{R}_{\scalebox{0.8}{$\C[N]$}}$ is a universal localization of $S_N$, then according to \cite[Theorem 4.7 \& Theorem 4.8]{Schofield85}
        \[
        \Tor_1^{\scalebox{0.8}{$S_N$}}(\mathcal{R}_{\scalebox{0.8}{$\C[N]$}}, \mathcal{R}_{\scalebox{0.8}{$\C[N]$}}) = 0
        \]
        as we wanted.

        Now we show the localization part. We recall that $N$ is defined as the fundamental group of the tree of amalgamated groups
        \[
        \ldots tAt^{-1} \ast_C A \ast_{t^{-1}Ct} t^{-1}At \ldots
        \]
        Set $H_1 := tAt^{-1} \ast_C A$ and for $n \geq 1$ let $H_{n+1}$ be the amalgamated free product of $H_n$ with the next right, respectively left, factor of the above tree of groups if $n$ is odd, respectively even. That is, $\{H_n : n \geq 1\}$ is a sequence of increasing subgroups of $N$ that covers $N$ via adding alternatively factors on the right and on the left. From the construction of the $\ast$-regular closure it follows that
        \[
        \mathcal{R}_{\scalebox{0.8}{$\C[N]$}} = \bigcup_{n \geq 1} \mathcal{R}_{\scalebox{0.8}{$\C[H_n]$}}.
        \]
        We show by induction on $n$ that $\mathcal{R}_{\scalebox{0.8}{$\C[H_n]$}}$ is a universal localization of the associated coproduct ring $R_{H_n}$ of the $\ast$-regular closures. The case $n = 1$ is \cref{thm: univ_local_amalg_prod}. So assume that $\mathcal{R}_{\scalebox{0.8}{$\C[H_n]$}}$ is a universal localization of $R_{H_n}$ with respect to the set of maps $\Sigma_n$ for some $n \geq 1$. Note that $H_n$ satisfies the strong Atiyah conjecture over $\C$ (see \cref{rem: sAc_arbitrary_amal_prod}). Suppose that $n$ is odd, the even case is proved with a symmetric argument. Again by \cref{thm: univ_local_amalg_prod}, we know that $\mathcal{R}_{\scalebox{0.8}{$\C[H_{n+1}]$}}$ is a universal localization of the coproduct ring
        \[
        \mathcal{R}_{\scalebox{0.8}{$\C[H_n]$}} \ast_{\scalebox{0.8}{$\mathcal{R}_{\scalebox{0.8}{$\C[t^{-(n+1)/2} C t^{(n+1)/2}]$}}$}} \mathcal{R}_{\scalebox{0.8}{$\C[t^{-(n+1)/2} A t^{(n+1)/2}]$}}
        \]
        with respect to the set of maps $\Sigma'_{n+1}$. However, by \cref{rem: K_0_R_G} and \cref{cor: K0_amalgamated}, the maps in $\Sigma'_{n+1}$ are between projective modules induced from finitely generated projectives $R_{H_{n+1}}$-modules. Thus by \cref{prop: local_transitive}
        \[
        \mathcal{R}_{\scalebox{0.8}{$\C[H_{n+1}]$}} = (R_{H_{{n+1}_{\scalebox{0.8}{$\Sigma_n$}}}})_{\Sigma_{n+1}'} = R_{H_{{n+1}_{\scalebox{0.8}{$\Sigma_{n+1}$}}}}
        \]
        where $\Sigma_{n+1} = \Sigma_n \cup \overline{\Sigma_{n+1}}$ and $\overline{\Sigma_{n+1}}$ is the set of induced maps that are stably associated to elements of $\Sigma'_{n+1}$. This finishes the induction. 
        
        Note that the above argument also shows that there is a natural directed system between the localizations since the sets in $\{ \Sigma_n : n\geq 1\}$ form an increasing union. Therefore, if we define $\Sigma = \cup_{n \geq 1} \Sigma_n$, we have
        \[
        \mathcal{R}_{\scalebox{0.8}{$\C[N]$}} = \bigcup_{n \geq 1} \mathcal{R}_{\scalebox{0.8}{$\C[H_n]$}} = \bigcup_{n \geq 1} R_{H_{n_{\scalebox{0.8}{$\Sigma_{n}$}}}} = S_{N_{\scalebox{0.8}{$\Sigma$}}}
        \]
        as we claimed. Finally, $\mathcal{R}_{\scalebox{0.8}{$\C[N]$}}$ is semisimple by \cref{prop: Atiyah_implies_semisimple} since $N$ satisfies the strong Atiyah conjecture over $\C$ according to \cref{rem: sAc_arbitrary_amal_prod}.
    \end{proof}

    From this statement, the remaining results follow.

    \begin{thm}\label{thm: SAC_HNN_localization_&_K0}
        Let $A$ be a group satisfying the strong Atiyah conjecture over $\C$ and $C$ a finite subgroup, and let $G = A \ast_C$ be their HNN-extension. Then the following holds:
        \begin{enumerate}[label=(\roman*)]
            \item $G$ satisfies the strong Atiyah conjecture over $\C$;
            \item $R_G$ is a hereditary ring and the $\SMRF$ induced from $\UO(G)$ via $\phi$ is faithful over finitely generated projective left $R_G$-modules;
            \item the map $K_0(\mathcal{R}_{\scalebox{0.8}{$\C[A]$}}) \rightarrow K_0(R_G)$ is surjective;
            \item $\mathcal{R}_{\scalebox{0.8}{$\C[G]$}}$ is the universal localization of all full maps over $R_G$ with respect to the $L^2$-dimension; and
            \item the map $K_0(\mathcal{R}_{\scalebox{0.8}{$\C[A]$}}) \rightarrow K_0(\mathcal{R}_{\scalebox{0.8}{$\C[G]$}})$ is surjective.
        \end{enumerate}
    \end{thm}
    \begin{proof}
        Now that we have established \cref{prop: vanishing_Tor1_HNN}, the proofs in the previous section work verbatim for the HNN setting. We highlight that \cref{thm: induced_submod} and \cref{prop: proj_mod_GR} are the only results that we used for the coproduct structure, but they hold in the greater generality of graph of rings with finite associated graph, and hence for the HNN of rings construction.
    \end{proof}

    \begin{thm}\label{thm: AAC_HNN_prod}
        Let $A$ be a group satisfying the algebraic Atiyah conjecture over $\C$ and $C$ a finite group. Then the HNN-extension $G = A \ast_C$ satisfies the algebraic Atiyah conjecture over $\C$.
    \end{thm}
    \begin{proof}
        By assumption, the algebraic Atiyah map of $A$ is surjective and $A$ satisfies the strong Atiyah conjecture over $\C$. Thus, by \cref{thm: SAC_HNN_localization_&_K0} the map $K_0(\mathcal{R}_{\scalebox{0.8}{$\C[A]$}}) \rightarrow K_0(\mathcal{R}_{\scalebox{0.8}{$\C[G]$}})$ is surjective. Composing the Atiyah map of $A$ with the former map shows that the Atiyah map of $G$ is surjective.
    \end{proof}

\section{Graphs of groups with finite edge groups} \label{sec: GG}

    In this section we combine the statements on amalgamated free products and HNN-extensions to prove the main results of the article. Recall that the upshots from \cref{sec: amalg_free_prod} and \cref{sec: HNN_ext} hold for a subfield $K \subseteq \C$ closed under complex conjugation. We begin with the strong Atiyah conjecture for the graph of groups construction. 

    \begin{proof}[Proof of \cref{thm: SAC_GG}]
        The strong Atiyah conjecture needs only to be checked over finitely presented modules, or by duality over matrices that have finitely many entries. Now, given a matrix there is a subgroup $H$ which is the fundamental group of a finite subgraph of groups of $\mathscr{G}_\Gamma$ satisfying that the group algebra $K[H]$ contains all the entries of the matrix. Hence, according to \cref{prop: preserve_dim_gps}, it suffices to check that $H$ satisfies the strong Atiyah conjecture over $K$. But this clearly follows from the combination of \cref{thm: SAC_amal_prod} and \cref{thm: SAC_HNN_localization_&_K0}.
    \end{proof}

    We continue with the universal localization. Given a graph of groups $\mathscr{G}_\Gamma = (G_v, G_e)$ with fundamental group $G$, one can always consider the associated graph of rings $\mathscr{RG}_\Gamma = (\mathcal{R}_{\scalebox{0.8}{$K[G_v]$}}, \mathcal{R}_{\scalebox{0.8}{$K[G_e]$}})$ where the ring homomorphisms are induced from the group homomorphisms of the graph of groups. Plainly, $K[G] \subseteq \mathscr{RG}_\Gamma$. Furthermore, as in the former sections, we have a commutative diagram of ring homomorphisms
    \[
    \begin{tikzcd}
        \mathscr{RG}_\Gamma \arrow[rr, "\phi"]&  & \UO(G)\\
        & K[G] \arrow[ur, hook] \arrow[lu, hook']&  
    \end{tikzcd}
    \]
    where the image of $\phi$ is once again contained in $\mathcal{R}_{\scalebox{0.8}{$K[G]$}}$. We now establish that $\mathcal{R}_{\scalebox{0.8}{$K[G]$}}$ is the universal localization of all full maps over $\mathscr{RG}_\Gamma$ with respect to the $L^2$-dimension.
    
    \begin{proof}[Proof of \cref{thm: GG_univ_local_&_K_0}]
        Since $\Gamma$ is countable, let $\{ \Gamma_n : n\geq 0\}$ be an exhausted nested sequence of finite subgraphs covering $\Gamma$ and such that $\Gamma_n$ has $n$ edges. Set $H_n$ to be the fundamental group of the subgraph of groups $\mathscr{G}_{\Gamma_n}$. From the construction of $\mathscr{RG}_{\Gamma}$, there is a natural direct limit structure
        \[
        \mathscr{RG}_{\Gamma} = \varinjlim_{n \geq 0} \mathscr{RG}_{\Gamma_n}
        \]
        where $\mathscr{RG}_{\Gamma_n}$ is the graph of rings associated to $H_n$. Similarly, the construction of the $\ast$-regular closure yields the next directed union
        \[
        \mathcal{R}_{\scalebox{0.8}{$K[G]$}} = \bigcup_{n \geq 0} \mathcal{R}_{\scalebox{0.8}{$K[H_n]$}}.
        \]
        The proof of \textit{(i)} is routine. First, recall that $K_0$-groups are compatible with direct limits, and so, we get
        \[
        K_0(\mathscr{RG}_{\Gamma}) = \varinjlim_{n \geq 0} K_0(\mathscr{RG}_{\Gamma_n}).
        \]
        Second, for finite graphs such as $\Gamma_n$, we can invoke \cref{thm: induced_submod}, which along with \cite[Lemma 4]{LinnellLuckSchick_OreCond}, shows the surjectivity of the map
        \[
        \oplus_{v \in V_n} K_0(\mathcal{R}_{\scalebox{0.8}{$K[G_v]$}}) \rightarrow K_0(\mathscr{RG}_{\Gamma_n})
        \]
        where $V_n$ stands for the subset of vertices of $\Gamma_n$. Thus, ranging over $n \geq 0$ shows \textit{(i)}. Observe that \textit{(ii)} follows straightforwardly from \textit{(i)} via a combination of \cref{prop: dimG_proj_faithful} and \cref{prop: preserve_dim_gps}.

        We show \textit{(iii)} and \textit{(iv)} together. First, we establish them for the finite subgrahps $\Gamma_n$, and then argue via direct limits to complete the proof. We work by induction on $n$. For $n = 0$, the map $\oplus_{v \in V_n} K_0(\mathcal{R}_{\scalebox{0.8}{$K[G_v]$}}) \rightarrow K_0(\mathcal{R}_{\scalebox{0.8}{$K[H_n]$}})$ is trivially surjective and $\mathcal{R}_{\scalebox{0.8}{$K[H_n]$}}$ is the universal localization of all full maps over $\mathscr{RG}_{\Gamma_n}$ with respect to the $L^2$-dimension as we are under the conditions of \cref{lem: full_vNr} according to \cref{prop: dimG_proj_faithful}. So assume now, for some $n \geq 0$, that $\mathcal{R}_{\scalebox{0.8}{$K[H_n]$}}$ is such universal localization of $\mathscr{RG}_{\Gamma_n}$ with respect to the set of maps $\Sigma_n$ and that the corresponding map of $K_0$-groups is surjective. We claim that $\mathcal{R}_{\scalebox{0.8}{$K[H_{n+1}]$}}$ is a universal localization of $\mathscr{RG}_{\Gamma_{n+1}}$. Indeed, if $\Gamma_{n+1}$ has an extra vertex, then $\mathscr{RG}_{\Gamma_{n+1}}$ is obtained from $\mathscr{RG}_{\Gamma_n}$ via a coproduct construction, and this case is identical to the one in the proof of \cref{prop: vanishing_Tor1_HNN}. So let us assume that $\Gamma_{n+1}$ differs from $\Gamma_n$ only in the edge $e$ with associated map $\widetilde{\sigma}$. By \cref{thm: SAC_HNN_localization_&_K0}, $\mathcal{R}_{\scalebox{0.8}{$K[H_{n+1}]$}}$ is a universal localization of the HNN ring
        \[
        \mathcal{R}_{\scalebox{0.8}{$K[H_n]_{\scalebox{0.8}{$\mathcal{R}_{\scalebox{0.8}{$K[G_e]$}}$}}$}} \langle t, t^{-1}; \widetilde{\sigma} \rangle
        \]
        with respect to a set of full maps that we denote by $\Sigma'_{n+1}$. Not only that, it also states that the map $K_0(\mathcal{R}_{\scalebox{0.8}{$K[H_n]$}}) \rightarrow K_0(\mathcal{R}_{\scalebox{0.8}{$K[H_n]_{\scalebox{0.8}{$\mathcal{R}_{\scalebox{0.8}{$K[G_e]$}}$}}$}} \langle t, t^{-1}; \widetilde{\sigma} \rangle)$ is surjective. Thus, by induction hypothesis and \cref{thm: induced_submod}, the collection of maps $\Sigma'_{n+1}$ is between stably induced finitely generated projectives over $\mathscr{RG}_{\Gamma_{n+1}}$. Hence, according to \cref{prop: local_transitive}
        \[
        \mathcal{R}_{\scalebox{0.8}{$K[H_{n+1}]$}} = (\mathscr{RG}_{\Gamma_{{n+1}_{\scalebox{0.8}{$\Sigma_n$}}}})_{\Sigma_{n+1}'} = \mathscr{RG}_{\Gamma_{{n+1}_{\scalebox{0.8}{$\Sigma_{n+1}$}}}}
        \]
        where $\Sigma_{n+1} = \Sigma_n \cup \overline{\Sigma_{n+1}}$ and $\overline{\Sigma_{n+1}}$ is the set of induced maps that are stably associated to elements of $\Sigma'_{n+1}$. This shows the claim. However, to finish the induction it remains to prove that the localization is done over all full maps with respect to the $L^2$-dimension. To this end, we show that we are under the conditions of \cref{thm: recogn_universal_localization}. Indeed, from \cref{thm: induced_submod} we deduce that $\mathscr{RG}_{\Gamma_{n+1}}$ is a hereditary ring since by hypothesis the vertex rings are semisimple according to \cref{prop: Atiyah_implies_semisimple}. Clearly, the map from $\mathscr{RG}_{\Gamma_{n+1}}$ to $\mathcal{R}_{\scalebox{0.8}{$K[H_{n+1}]$}}$ is epic. In addition, from a combination of \cref{prop: proj_mod_GR} and \cite[Lemma 4]{LinnellLuckSchick_OreCond}, the map
        \[
        \oplus_{v \in V_{n+1}} K_0(\mathcal{R}_{\scalebox{0.8}{$K[G_v]$}}) \rightarrow K_0(\mathscr{RG}_{\Gamma_{n+1}})
        \]
        is surjective, so the induced $\SMRF$ on $\mathscr{RG}_{\Gamma_{n+1}}$ is faithful over finitely generated projective modules by \cref{prop: dimG_proj_faithful} and \cref{prop: preserve_dim_gps}. Moreover, $H_{n+1}$ satisfies the strong Atiyah conjecture over $K$ by \cref{thm: SAC_GG}, and hence, $\mathcal{R}_{\scalebox{0.8}{$K[H_{n+1}]$}}$ is a semisimple ring according to \cref{prop: Atiyah_implies_semisimple}. Finally, from \cite[Theorem 4.7 \& Theorem 4.8]{Schofield85} it holds that 
        \[
        \Tor_1^{\scalebox{0.8}{$\mathscr{RG}_{\Gamma_{n+1}}$}}(\mathcal{R}_{\scalebox{0.8}{$K[H_{n+1}]$}} , \mathcal{R}_{\scalebox{0.8}{$K[H_{n+1}]$}}) = 0
        \]
        $\mathcal{R}_{\scalebox{0.8}{$K[H_{n+1}]$}}$ being already a universal localization of $\mathscr{RG}_{\Gamma_{n+1}}$. Summing up, we can apply \cref{thm: recogn_universal_localization} to complete the localization part of the induction. With regard to the map of $K_0$-groups, we are able to use \cref{prop: project_mod_localization}. Therefore, the map $K_0(\mathscr{RG}_{\Gamma_{n+1}}) \rightarrow K_0(\mathcal{R}_{\scalebox{0.8}{$K[H_{n+1}]$}})$ is surjective, and thus the induction follows from \textit{(i)}.

        To conclude, we make explicit the direct limit arguments. The discussion above shows that there is a natural directed system between the localizations since the sets in $\{ \Sigma_n : n \geq  0 \}$ form an increasing union. Therefore, if we define $\Sigma = \cup_{n \geq 0} \Sigma_n$, we have
        \[
        \mathcal{R}_{\scalebox{0.8}{$K[G]$}} = \bigcup_{n \geq 0} \mathcal{R}_{\scalebox{0.8}{$K[H_n]$}} = \bigcup_{n \geq 0} \mathscr{RG}_{\Gamma_{n_{\scalebox{0.8}{$\Sigma_{n}$}}}} = \mathscr{RG}_{\Gamma_{\scalebox{0.8}{$\Sigma$}}}.
        \]
        Note that all the maps in $\Sigma$ are full with respect to the $L^2$-dimension by \cref{lem: invertiv_is_full}. Hence, it only remains to show that all full maps over $\mathscr{RG}_{\Gamma}$ are invertible over $\mathcal{R}_{\scalebox{0.8}{$K[G]$}}$. Recall that there is an explicit construction for universal localizations via matrices (see \cref{subsec: univ_local}). So let $a$ be the endomorphism associated to a full map $\alpha$ over finitely generated projective $\mathscr{RG}_{\Gamma}$-modules. Since matrices have only finitely many entries and $\mathscr{RG}_{\Gamma} = \varinjlim_{n \geq 0} \mathscr{RG}_{\Gamma_n}$, we can assume that $a$ is realized as a map between finitely generated projective $\mathscr{RG}_{\Gamma_n}$-modules for some $n \geq 0$. This map has to be full over $\mathscr{RG}_{\Gamma_n}$ since $\alpha$ is full, and hence, it must be invertible over $\mathcal{R}_{\scalebox{0.8}{$K[H_n]$}}$, and consequently over $\mathcal{R}_{\scalebox{0.8}{$K[G]$}}$. Therefore, $(\mathcal{R}_{\scalebox{0.8}{$K[G]$}}, \phi)$ is the universal localization of $\mathscr{RG}_{\Gamma}$ at all full maps with respect to the $L^2$-dimension, which finishes the proof of \textit{(iii)}. Finally, \textit{(iv)} follows from the direct limit
        \[
        K_0(\mathcal{R}_{\scalebox{0.8}{$K[G]$}}) = \varinjlim_{n \geq 0} K_0(\mathcal{R}_{\scalebox{0.8}{$K[H_n]$}})
        \]
        and the surjectivity of the map $\oplus_{v \in V_n} K_0(\mathcal{R}_{\scalebox{0.8}{$K[G_v]$}}) \rightarrow K_0(\mathcal{R}_{\scalebox{0.8}{$K[H_n]$}})$ showed by induction for every $n \geq 0$.
    \end{proof}

\section{Class \texorpdfstring{$\mathcal{T}$} \ \ and virtually-\{locally indicable\} groups} \label{sec: class_T}

    In this section we study the \textit{class $\mathcal{T}$} defined in the introduction. It is constructed in the spirit of Linnell's class $\mathcal{C}$ but we allow the graph of groups construction instead of extensions by elementary amenable groups: the class $\mathcal{T}$ is the smallest class of groups which
    \begin{enumerate}[label=(\roman*)]
        \item contains all finite groups;
        \item is closed under directed unions; and
        \item satisfies $G \in \mathcal{T}$ whenever $G$ is the fundamental group of a graph of groups in $\mathcal{T}$.
    \end{enumerate}
    We believe that the class $\mathcal{T}$ is a reasonable candidate to enjoy the universal localization property as well as the Atiyah conjectures. Nevertheless, this goal is beyond the scope of this article and we restrict to the following subclass. We recall that a group $G$ is \textit{virtually-\{locally indicable\}} if there is some finite index subgroup $H$ which is locally indicable, that is, each of its non-trivial finitely generated subgroups maps homomorphically onto $\Z$.

    \begin{defn}
        Let $\mathcal{T}_{\scalebox{0.7}{$\mathcal{VLI}$}}$ denote the smallest class of groups which
        \begin{enumerate}[label=(\roman*)]
            \item contains all finite groups;
            \item is closed under directed unions; and
            \item satisfies $G \in \mathcal{T}_{\scalebox{0.7}{$\mathcal{VLI}$}}$ whenever $G$ is the fundamental group of a graph of groups in $\mathcal{T}_{\scalebox{0.7}{$\mathcal{VLI}$}}$ and $G$ is in addition virtually-\{locally indicable\}.
        \end{enumerate}
    \end{defn}

    We shall see that every group in the class $\mathcal{T}_{\scalebox{0.7}{$\mathcal{VLI}$}}$ does satisfy the strong, algebraic and center-valued Atiyah conjecture, and moreover, that its $\ast$-regular closure is a specific universal localization of the group algebra. The main reason to choose (virtually-) locally indicable groups, is due to the study of the graph of groups construction and locally indicable groups done in \cite{FisherPeralta_Kaplansky'sZD3mfld} by Sam P. Fisher and the author. Concretely, given a graph of groups $\mathscr{G}_\Gamma = (G_v, G_e)$ with fundamental group $G$ and associated graph of rings $\mathscr{RG}_\Gamma = (\mathcal{R}_{\scalebox{0.8}{$\C[G_v]$}}, \mathcal{R}_{\scalebox{0.8}{$\C[G_e]$}})$, \cite[Lemma 7.6]{FisherPeralta_Kaplansky'sZD3mfld} establishes that $\mathcal{R}_{\scalebox{0.8}{$\C[G]$}}$ is the universal localization over all full matrices of $\mathscr{RG}_\Gamma$ with respect to the $L^2$-dimension. For the purpose of showing \cref{thm: class_TVLI}, we exhibit a stratified description of the class $\mathcal{T}_{\scalebox{0.7}{$\mathcal{VLI}$}}$ that enables us to proceed by induction. Namely, for each ordinal $a$, the class of groups $\mathcal{Y}_a$ is defined inductively by $\mathcal{Y}_0 = $ \{all finite groups\} and $\mathcal{Y}_{a+1} = $ \{virtually-\{locally indicable\} groups that are the fundamental group of a graph of groups with vertex and edge groups locally in $\mathcal{Y}_a$ \} and $\mathcal{Y}_b = \cup_{ a < b} \mathcal{Y}_a$ if $b$ is a limit ordinal, where a group is locally in $\mathcal{Y}_a$ if each of its non-trivial finitely generated subgroups lies in $\mathcal{Y}_a$. Set $\mathcal{Y} = \cup_{a \geq 0} \mathcal{Y}_a$, then we have:

    \begin{lem}\label{lem: Ya_class_TVLI} The class $\mathcal{Y}$ is precisely the class $\mathcal{T}_{\scalebox{0.7}{$\mathcal{VLI}$}}$. 
    \end{lem}
    \begin{proof}
        Clearly, $\mathcal{Y} \subseteq \mathcal{T}_{\scalebox{0.7}{$\mathcal{VLI}$}}$. The other direction also holds since $\mathcal{Y}$ contains finite groups, is closed under directed unions and is closed under the graph of groups constructions provided that the fundamental group is virtually-\{locally indicable\}.
    \end{proof}

    The argument for \cref{thm: class_TVLI} is divided into two main parts. First, we show it for graph of groups that are virtually-\{universal localizations\}, and then we use some routine techniques to reduce to this case. We recall that given $K \subseteq \C$ a field closed under complex conjugation, $\Sigma(K,G)$ stands for the set of matrices over $K[G]$ that are invertible over $\UO(G)$.

    \begin{prop}\label{prop: virt_localization_and_GG}
        Let $G$ be a fundamental group of a finite graph of groups $\mathscr{G}_\Gamma = (G_v, G_e)$, $H$ a finite index normal subgroup of $G$ and $K \subseteq \C$ a field closed under complex conjugation. Assume that the following conditions hold:
        \begin{enumerate}[label=(\roman*)]
            \item Each $G_v$ and $G_e$ satisfies the strong Atiyah conjecture over $K$;
            \item each $\mathcal{R}_{\scalebox{0.8}{$K[G_v]$}}$ and $\mathcal{R}_{\scalebox{0.8}{$K[G_e]$}}$ is the universal localization of $K[G_v]$ and $K[G_e]$ over $\Sigma(K,G_v)$ and $\Sigma(K, G_e)$, respectively; and
            \item the ring $\mathcal{R}_{\scalebox{0.8}{$K[H]$}}$ is semisimple and a universal localization of $K[H]$ over $\Sigma(K,H)$.
        \end{enumerate}
        Then $G$ satisfies the strong Atiyah conjecture over $K$ and $\mathcal{R}_{\scalebox{0.8}{$K[G]$}}$ is the universal localization of $K[G]$ over $\Sigma(K,G)$. Moreover, if in addition each $G_v$ and $G_e$ satisfies the algebraic Atiyah conjecture over $K$, so does $G$.
    \end{prop}
    \begin{proof}
        To begin with, the subring $\mathcal{R}_{\scalebox{0.8}{$K[H]$}} \ast G/H \subseteq \mathcal{R}_{\scalebox{0.8}{$K[G]$}}$ is semisimple being $G/H$ a finite group and $\mathcal{R}_{\scalebox{0.8}{$K[H]$}}$ a semisimple ring of characteristic $0$; see, for instance, \cite[Lemma 8.12 (iv)]{ReichThesis}. Thus, $\mathcal{R}_{\scalebox{0.8}{$K[H]$}} \ast G/H$ is a $\ast$-regular ring containing $K[G]$. So, by definition of the $\ast$-regular closure,
        $\mathcal{R}_{\scalebox{0.8}{$K[G]$}} = \mathcal{R}_{\scalebox{0.8}{$K[H]$}} \ast G/H$. Moreover, by \cite[Lemma 4.5] {LinnellDivRings93} $\mathcal{R}_{\scalebox{0.8}{$K[G]$}}$ is a universal localization of $K[G]$ over $\Sigma(K,H)$. We claim that $K[G] \hookrightarrow \mathcal{R}_{\scalebox{0.8}{$K[G]$}}$ is universal $\Sigma(K,G)$-inverting. Let $A$ be a matrix over $K[G]$ that becomes invertible over $\UO(G)$. Then $A$ as a matrix over $\mathcal{R}_{\scalebox{0.8}{$K[G]$}}$ is also invertible (see, for instance, \cite[Proposition 13.15]{ReichThesis}). So $\mathcal{R}_{\scalebox{0.8}{$K[G]$}}$ is the universal localization of $K[G]$ over $\Sigma(K,G)$.
        
        On the other hand, we have by hypothesis that the graph of rings $\mathscr{RG}_{\Gamma}$ is a universal localization of $K[G]$. Thus, invoking \cite[Theorem 4.7 \& Theorem 4.8]{Schofield85} it holds that
        \[
        \Tor_1^{\scalebox{0.8}{$\mathscr{RG}_{\Gamma}$}}(\mathcal{R}_{\scalebox{0.8}{$K[G]$}}, \mathcal{R}_{\scalebox{0.8}{$K[G]$}}) = \Tor_1^{\scalebox{0.8}{$K[G]$}}(\mathcal{R}_{\scalebox{0.8}{$K[G]$}}, \mathcal{R}_{\scalebox{0.8}{$K[G]$}}) = 0.
        \]
        From this vanishing, the results follow. If the vertex and edge groups satisfy the strong Atiyah conjecture, we can repeat the argument of the proof of \cref{thm: SAC_amal_prod} to conclude that $G$ satisfies the strong Atiyah conjecture. Note that $\mathcal{R}_{\scalebox{0.8}{$K[G]$}}$ is also the universal localization of all full maps over $\mathscr{RG}_{\Gamma}$ with respect to the $L^2$-dimension because the proof of \cref{thm: univ_local_amalg_prod} can be reproduced via \cref{thm: recogn_universal_localization}. Therefore, if in addition the vertex and edge groups satisfy the algebraic Atiyah conjecture, we proceed as in the proof of \cref{thm: AAC_amal_prod} invoking \cref{prop: project_mod_localization} to get that $G$ satisfies the algebraic Atiyah conjecture.
    \end{proof}
    
    Now we are ready to show the analog of \cite[Theorem 1.5]{LinnellDivRings93} for the class $\mathcal{T}_{\scalebox{0.7}{$\mathcal{VLI}$}}$.

    \begin{proof}[Proof of \cref{thm: class_TVLI}]
        According to \cref{lem: Ya_class_TVLI}, it suffices to show \textit{(i)} and \textit{(ii)} for each $\mathcal{Y}_a$ with $a \geq 0$. We argue by induction on $a$. For $a = 0$, \textit{(i)} and \textit{(ii)} follow trivially because $\mathcal{Y}_0$ is the class of finite groups. Since $\Sigma(K,G)$-localizations and Atiyah conjectures behave well under direct limit constructions, we may assume that $a = b + 1$ for some ordinal $b$. Let $G$ be the fundamental group of a graph of groups $\mathscr{G}_\Gamma = (G_v, G_e)$ where $G_v$ and $G_e$ are locally in $\mathcal{Y}_b$ and such that $G$ is virtually-\{locally indicable\}. Note that $G$ is a directed union of subgroups in $\mathcal{Y}_a$ with finite associated graph of groups. So we may also assume that $\Gamma$ is a finite graph to prove \textit{(i)} and \textit{(ii)} for $G$. Let $H$ be the locally indicable subgroup with finite index in $G$, which without loss of generality is normal in $G$. We show that $G$ and $H$ are under the conditions of \cref{prop: virt_localization_and_GG}. 
        
        First, from the action of $G$ in its Bass--Serre tree, we obtain a graph of groups decomposition for $H$, say $\mathscr{H}_{\Gamma_H} = (H_v, H_e)$, where the vertex and edge groups inherit the property of being locally in $\mathcal{Y}_b$. Once again, a direct limit argument shows that all $H_v$ and $H_e$ satisfy \textit{(i)} and \textit{(ii)}. Thus the graph of rings $\mathscr{RH}_{\Gamma_H}$ is a universal localization over a set of matrices of $K[H]$ that become invertible over $\UO(H)$. Furthermore, $H$ is locally indicable, so $\mathcal{R}_{\scalebox{0.8}{$K[H]$}}$ is a division ring by \cite[Theorem 1.1]{JaikinLopez_Atiyah}. Not only that, according to \cite[Lemma 7.6]{FisherPeralta_Kaplansky'sZD3mfld} $\mathcal{R}_{\scalebox{0.8}{$K[H]$}}$ is the universal localization over the set of all full matrices of $\mathscr{RH}_{\Gamma_H}$, which we denote by $\Sigma_H$. Therefore, we can apply \cref{prop: local_transitive} to get that $\mathcal{R}_{\scalebox{0.8}{$K[H]$}}$ is a universal localization of $K[H]$ at some set of matrices $\Sigma$ that become invertible over $\UO(H)$. We claim that $K[H] \hookrightarrow \mathcal{R}_{\scalebox{0.8}{$K[H]$}}$ is universal $\Sigma(K,H)$-inverting. Indeed, if $A$ is a matrix over $K[H]$ that becomes invertible over $\UO(H)$, then $A$ is a full matrix over $\mathscr{RH}_{\Gamma_H}$ by \cref{lem: invertiv_is_full}, and so $A$ is already in $\Sigma_H$. So $\mathcal{R}_{\scalebox{0.8}{$K[H]$}}$ is the universal localization of $K[H]$ over $\Sigma(K,H)$, and we can invoke \cref{prop: virt_localization_and_GG} to conclude that $G$ satisfies \textit{(i)} and \textit{(ii)}. This completes the induction, and hence the proof.
    \end{proof}

    To end the section, we focus on $K_1(\mathcal{R}_{\scalebox{0.8}{$K[G]$}})$. Let us recall the definition of $K_1^w(R[G])$ given in \cite[Definition 3.1 \& Remark 3.2]{LuckLinnell_localization} for a ring $R$ with $\Z \subseteq R \subseteq \C$. Elements $[A]$ of $K_1^w(R[G])$ are given by square $n$-matrices $A$ over $R[G]$ which are not necessarily invertible but for which the operator $r_A^{(2)}: \ell^2 (G)^n \rightarrow \ell^2 (G)^n$ given by right multiplication with $A$ is a \textit{weak isomorphism}, that is, it is injective and has dense image. We require for such square matrices $A, B$ the following relations in $K_1^w(R[G])$:
    \[
    [AB] = [A] \cdot [B]; \quad \left[ \begin{pmatrix}
                                    A & \ast \\
                                    0 & B
                                \end{pmatrix} \right] = [A] \cdot [B].
    \]
    The usual $K_1$-group only considers invertible matrices subjected to the same relations.

    \begin{proof}[Proof of \cref{thm: Linnell_Luck_localization}]
        In \cref{thm: class_TVLI} we already showed that $\mathcal{R}_{\scalebox{0.8}{$F[G]$}}$ is a semisimple ring and the universal localization of $F[G]$ over $\Sigma(F, G)$. Therefore the right hand side isomorphisms follows from the next chain of isomorphisms
        \[
        K_1(\mathcal{R}_{\scalebox{0.8}{$F[G]$}}) \cong \prod_{i=1}^l K_1(\Mat_{n_i}(\mathcal{D}_i)) \cong \prod_{i=1}^l K_1(\mathcal{D}_i) \cong \prod_{i=1}^l \mathcal{D}_i^{\times}/[\mathcal{D}_i^{\times}, \mathcal{D}_i^{\times}]
        \]
        where $\prod_{i=1}^l \Mat_{n_i}(\mathcal{D}_i)$ is the Artin--Wedderburn decomposition associated to $\mathcal{R}_{\scalebox{0.8}{$F[G]$}}$ and the last isomorphism is given by the Dieudonn\'e determinant \cite[Corollary 2.2.6]{Rosenberg}.

        For the left part, we follow \cite[Remark 3.2]{LuckLinnell_localization}. Given a square $n$-matrix $A$ over $R[G]$, according to \cite[Theorem 6.24 \& Theorem 8.22(5)]{Luck02} the operator $r_A^{(2)}: \ell^2 (G)^n \rightarrow \ell^2 (G)^n$ is a weak isomorphism if and only if $A$ is invertible in $\UO(G)$. But since $\mathcal{R}_{\scalebox{0.8}{$F[G]$}}$ is a von Neumman regular ring, $A$ is invertible in $\UO(G)$ if and only if it is invertible in $\mathcal{R}_{\scalebox{0.8}{$F[G]$}}$ (see, for instance, \cite[Proposition 13.15]{ReichThesis}). Therefore the proof ends applying \cite[Theorem 1.2]{LuckLinnell_localization}.
    \end{proof}

\section{Further comments and conjectures}\label{sec: comments_questions}

    In the light of the arguments above, there is no other restriction to extend the graph of groups construction to arbitrary edge groups that satisfy the strong Atiyah conjecture over $K \subseteq \C$ a field closed under complex conjugation, rather than the vanishing of $\Tor_1^{\scalebox{0.8}{$\mathscr{RG}_\Gamma$}}(\UO(G),\UO(G))$. One possible approach is to generalize \cref{lem: amalg_prod_intersection_augment_ideal} to an arbitrary group $C$, though it seems that whenever $C$ is infinite the several topologies of bounded operators play a different role. On the other hand, it would be insightful to prove this vanishing via algebraic methods, and perhaps the fact that $\mathcal{R}_{\scalebox{0.8}{$K[C]$}}$ is a universal localization of $K[C]$ could be important to tackle this question. Thus we propose the following conjecture.

    \begin{conj} \label{conj: vanishing_Tor_1}
        Let $A$, $B$ and $C$ be three groups and let $G=A \ast_C B$ be their amalgamated free product. For $K \subseteq \C$ a field closed under complex conjugation, set $R_{G}:= \mathcal{R}_{\scalebox{0.8}{$K[A]$}} \ast_{\scalebox{0.8}{$\mathcal{R}_{\scalebox{0.8}{$K[C]$}}$}} \mathcal{R}_{\scalebox{0.8}{$K[B]$}}$. Assume in addition that $\mathcal{R}_{\scalebox{0.8}{$K[C]$}}$ is a univeral localization of $K[C]$. Then
        \[
        \Tor_1^{R_{G}}(\UO(G),\UO(G)) = 0.
        \]
    \end{conj}

    The Atiyah conjectures and the universal localization construction are closely related as we have shown in this article. Recall that for a group $G$ in Linnell's class $\mathcal{C}$ with finite $\lcm(G)$, $\mathcal{R}_{\scalebox{0.8}{$K[G]$}}$ is a universal localization of $K[G]$. This is also true for some locally indicable groups such as torsion-free one-relator groups or the fundamental group of a graph of groups of free groups and infinite cyclic edge groups. We want to go further and state the next conjecture.

    \begin{conj} \label{conj: class_T}
        Every group $G$ in the class $\mathcal{T}$ with finite $\lcm(G)$ satisfies the strong, algebraic and center-valued Atiyah conjecture over any subfield $K \subseteq \C$ closed under complex conjugation and $\mathcal{R}_{\scalebox{0.8}{$K[G]$}}$ is the universal localization of $K[G]$ over $\Sigma(K,G)$.
    \end{conj}

    A positive answer to \cref{conj: vanishing_Tor_1} would led to a generalization of \cref{thm: class_TVLI} in which the groups can be taken in the whole class $\mathcal{T}$, and in particular, it would confirm \cref{conj: class_T}.

\bibliographystyle{alpha}
\bibliography{bib}

\end{document}